\documentclass[oneside]{article}
\usepackage{amsmath}
\usepackage{amssymb}
\usepackage{bbm}
\usepackage[margin=3cm]{geometry}
\usepackage{hyperref}
\usepackage{tikz}
\usepackage{braket}

\newtheorem{thm}{Theorem}[section]
\newtheorem{cor}[thm]{Corollary}
\newtheorem{prop}[thm]{Proposition}

\newtheorem{lemma}[thm]{Lemma}

\newtheorem{theorem}[thm]{Theorem}
\newtheorem{remark}[thm]{Remark}
\newtheorem{proposition}[thm]{Proposition}
\newtheorem{definition}[thm]{Definition}
\newtheorem{corollary}[thm]{Corollary}

\newenvironment{proof}{{\bf Proof:}}{\hfill$\square$\vskip.5cm}
\newenvironment{proofof}{}{\hfill$\square$\vskip.5cm}
\newcommand{\R}{\mathbb{R}}
\newcommand{\N}{\mathbb{N}}
\newcommand{\C}{\mathbb{C}}
\newcommand{\E}{\mathbf{E}}

\newcommand{\Hil}{\mathcal{H}}

\title{Phase Uniqueness for the Mallows Measure on Permutations}
\author{Shannon Starr$^{1}$ and Meg Walters$^{2}$\\
\small
${}^1$ University of Alabama at Birmingham, Applied Mathematics, Birmingham, AL 35294--1170\quad \url{slstarr@uab.edu}\\
\small
${}^2$ University of Rochester, Department of Mathematics, Rochester, NY 14627\quad \url{meg.walters@rochester.edu}
}

\date{September 12, 2017}

\begin{document}

\maketitle

\begin{abstract}
\setcounter{section}{0}
For a positive number $q$ the Mallows measure on the symmetric group is the 
probability measure on $S_n$  such that $P_{n,q}(\pi)$ is proportional to $q$-to-the-power-$\mathrm{inv}(\pi)$
where $\mathrm{inv}(\pi)$ equals the number of inversions: $\mathrm{inv}(\pi)$ equals the number of pairs $i<j$ such that $\pi_i>\pi_j$.
One may consider this as a mean-field model from statistical mechanics.
The weak large deviation principle may replace the Gibbs variational principle
for characterizing equilibrium measures.
In this sense, we prove absence of phase transition, i.e., phase uniqueness.
\end{abstract}

\section{Introduction} 

The Mallows measure on permutations is a non-uniform measure which may be motivated in various ways.
It arises in non-parametric statistics \cite{Mallows}
\begin{equation}
\label{eq:Pdef}
\forall \pi \in S_n\, ,\quad
P_{n,q}(\pi)\, =\, Z_{n,q}^{-1}\,  
q^{\operatorname{inv}(\pi)}\, ,\ \text{ where }\
\operatorname{inv}(\pi)\, \stackrel{\mathrm{def}}{:=}\,  
\sum_{i=1}^{n-1} \sum_{j=i+1}^{n} \mathbf{1}_{(-\infty,0)}(\pi_j-\pi_i)\, ,
\end{equation}
for some parameter $q \in (0,\infty)$.

A good recent review of several important examples of non-uniform measures is \cite{Mukherjee}.
In addition, in that reference, Mukherjee considered thermodynamic limits, and he derived large deviation principles.
That is the starting point for us.

The Mallows model has also been studied for other reasons.
Diaconis and Ram showed a connection to the Hecke algebra \cite{DiaconisRam}.
The Mallows measure is a measure on permutations. 
But it is also closely related to the ``blocking measure'' which are invariant measures of the asymmetric
exclusion process (ASEP).
In fact the ASEP may be viewed as a projection of a biased card shuffling algorithm
introduced by Diaconis and Ram.
This was exploited by Benjamini, Berger, Hoffman and Mossel \cite{BBHM},
using David Bruce Wilson's height functions \cite{Wilson} to bound the mixing time for the card shuffling model,
starting from the mixing time for the ASEP.

At a simpler level, one may try to use information about the ASEP invariant measures to gain information
about the Mallows measure.
The ASEP invariant measures are well-known. See, for example, Chapter VIII, Section 5 of \cite{Liggett}.
This is a well-known approach, following Wilson \cite{Wilson}.

In the present note, we consider the large deviation principle for a continuous version of the Mallows model,
$\mu_{n,\beta}$ on $([0,1]^2)^n$ such that
\begin{equation}
\label{eq:H1}
\begin{gathered}
d\mu_{n,\beta}((x_1,y_1),\dots,(x_n,y_n))\,
=\, Z_n(\beta)^{-1} e^{-\beta H_n((x_1,y_1),\dots,(x_n,y_n))}\, ,\\
H_n((x_1,y_1),\dots,(x_n,y_n))\,
=\, \frac{1}{n-1}\, \sum_{i=1}^{n-1} \sum_{j=i+1}^{n} h((x_i,x_j),(y_i,y_j))\, \\
h((x_1,y_1),(x_2,y_2))\, =\, \mathbf{1}_{(-\infty,0)}((x_1-x_2)(y_1-y_2))\, .
\end{gathered}
\end{equation}
Mukherjee found the 
large deviation rate function $\mathcal{I}_{\beta} : \mathcal{M}_{+,1}([0,1]^2)$ for the empirical measure of $((x_1,y_1),\dots,(x_n,y_n))$
on $[0,1]^2$
$$
\mathfrak{m}^{(n)}_{((x_1,y_1),\dots,(x_n,y_n))}(\bullet)\,
=\, \frac{1}{n}\, \sum_{k=1}^{n} \mathbf{1}_{\bullet}((x_k,y_k))\, .
$$
We start with his formula for the rate function, and  we show that there is a unique optimizer.
There is a straightforward connection between $\mu_{n,\beta}$ and $P_{n,q}$.
Therefore, this gives a direct probabilistic method to find the weak limit law for $\mu_{n,\beta}$.

\subsection{Discussion of proof technique and relation to known results}
Following the approach suggested by the height functions, we consider the 4-square problem.
Given $\theta \in (0,1)$ define $L_1(\theta) = [0,\theta]$ and $L_2(\theta) = (\theta,1]$.
Given $\theta_1,\theta_2 \in (0,1)$ define $\Lambda_{ij}(\theta_1,\theta_2) = L_i(\theta_1) \times L_j(\theta_2)$ in $[0,1]^2$.
Then, finally, given $t_{11},t_{12},t_{21},t_{22} \geq 0$ such that $t_{11}+t_{12}+t_{21}+t_{22}=1$,
define
$$
W_{\theta_1,\theta_2}(t_{11},t_{12},t_{21},t_{22})\,
=\, \{\nu \in \mathcal{M}_{+,1}([0,1]^2)\, :\, \forall i,j \in \{1,2\}\, ,\ 
\nu(\Lambda_{ij}(\theta_1,\theta_2))=t_{ij}\}\, .
$$
We give an explicit formula for $\mathcal{I}_{\beta}(W_{\theta_1,\theta_2}(t_{11},t_{12},t_{21},t_{22}))$.
Intuitively, this is the simplest possible problem one can consider,
starting from Mukherjee's formula for $\mathcal{I}_{\beta}$.
Moreover, for each $(\theta_1,\theta_2) \in (0,1)^2$, there is a unique choice of $t_{ij}^*(\theta_1,\theta_2)$ maximizing this formula.
Moreover, 
$$
R_{\beta}(\theta_1,\theta_2)\, =\, t_{11}^*(\theta_2,\theta_2)\, ,
$$
defined for each $\theta_1,\theta_2 \in [0,1]$ (and extended continuously at the boundary)
corresponds to the joint cumulative distribution function of a measure $\rho_{\beta} \in \mathcal{M}_{+,1}([0,1]^2)$: $\rho_{\beta}([0,\theta_1]\times [0,\theta_2]) = R_{\beta}(\theta_1,\theta_2)$.

As a corollary, elementary results imply that $\mathfrak{m}^{(n)}_{((x_1,y_1),\dots,(x_n,y_n))}$
converge in distribution to the non-random measure $\rho_{\beta}$,
when for each $n$ we have $((x_1,y_1),\dots,(x_n,y_n))$ distributed according to $\rho_{\beta}$.
Then, due to the connection to the Mallows measure, the same result holds
when applied to $\mathfrak{m}^{(n)}_{((x_1,y_1),\dots,(x_n,y_n))}$, if, for each $n$
we define $x_k = k/n$ and $y_k = \pi_k/n$ for $k=1,\dots,n$, where we select a random 
permutation $\pi \in S_n$, distributed according to $P_{n,q_n}$, as long
as $(q_1,q_2,\dots)$ is a sequence such that $\lim_{n \to \infty} n(1-q_n) = \beta$.
Note that $\beta \in \R$ is fixed.
(For $\beta<0$ this means that $q_n$ is slightly greater than $1$ instead of less than $1$,
as it would be for $\beta>0$.)
So this result also gives a simpler, more direct proof of an old result from \cite{Starr},
which had previously been proved by an obscure method.

\subsubsection{The Mallows model is a frustration free, mean-field model}

The Hamiltonian in (\ref{eq:H1}) is a mean-field Hamiltonian.
Mean-field Hamiltonians have the property that considering a subsystem,
the inverse-temperature $\beta$ needs to be rescaled because of the explicit
dependence of $H_n$ on $n$.
The consideration of a sub-system is sometimes known as the cavity method for 
complicated problems, which are most amenable to inductive analysis, removing one
particle at a time.
See for instance \cite{MPV} as an indication of the physics approach to this method
or \cite{Tal} for the mathematical side.
Temperature renormalization means that if there is an explicit formula for the optimizer
of the Mallows model on $[0,1]^2$ in the thermodynamic limit,
then restricting attention to the sub-squares $\Lambda_{ij}(\theta_1,\theta_2)$,
the restriction of the measure to these sets may be an optimizer
for different choices of $\beta$, due to dilution on the sub-squares.
(We will refer to the $\Lambda_{ij}(\theta_1,\theta_2)$'s as ``sub-squares''
even though they are rectangles.)

Of course, since the model is a mean-field model, there is an interaction between all particles
in $[0,1]^2$, including between different sub-squares.
But then there is a special symmetry of the model.
In two dimensions, it is common to see conformally invariant models,
which is the symmetry one finds from local rotational symmetry
as well as dilation covariance.
For the Mallows model, instead of $\mathrm{SO}(2)$ symmetry
the group which leaves the model invariant is the group of hyperbolic
rotations $\mathrm{SO}^+(1,1)$, because 
one may rescale the two dimensions as long as the area remains fixed,
in what is sometimes known as a ``squeeze transformation.''

Moreover, one can factorize the degrees of freedom of a measure on a square
by first considering its $x$ and $y$ marginals, and then considering the measure
in the square with those marginals, which we call the ``coupling measure.''
In the Mallows model, the different sub-squares only interact through their marginals.
We think of the $x$ and $y$ marginals as data living on the boundary of each square,
and the coupling measure as data interior.
Then the choice of the coupling measure in the interior of each subsquare is not affected
by the choice of the boundary marginal measures.
So for each sub-square, 
the coupling measure is an un-restricted optimizer of the rate function $I_{\beta_{ij}}$,
but at a diluted value of the inverse-temperature $\beta_{ij}$.
In this sense, the problem is ``frustration free.''
Moreover, this shifts the problem to determining the optimal choices of $\beta_{11},\beta_{12},\beta_{21},\beta_{22}$, which are in turn explicit functions of $t_{11},t_{12},t_{21},t_{22}$
due to the explicit formula for the pressure $p(\beta) = \lim_{n \to \infty} n^{-1} \ln(Z_n(\beta))$.
So, in this sense the data $R_{\beta}(\theta_1,\theta_2)$
may be deduced just from the general formula for the pressure $p : \R \to \R$.

\subsubsection{Discrete symmetry and integrability}
The main point of this paper is to give a simple proof of the uniqueness of the optimizer of $I_{\beta}$,
exploiting the symmetry just described in the continuum limit.
But there is a related symmetry of the Mallows measure even for finite $\beta$.
For example, it is related to the Fisher-Yates-Knuth algorithm for perfectly simulating
permutations.
This is also related to the powerful bounds and approximations of Bhatnagar and Peled \cite{BhatnagarPeled}.
The symmetry may also be deduced from the ``height-function''
approach of Wilson \cite{Wilson},
which was exploited by \cite{BBHM},
relating the biased card-shuffling algorithm of Diaconis and Ram \cite{DiaconisRam}
to the Markov chain projection, which is the asymmetric exclusion process (ASEP).
(Of course, in the present paper, we only consider invariant measures, not the actual stochastic
dynamics, which is at least one level higher.)

The ASEP possesses reversible invariant measures.
Using these, one may make a 
similarity transform of the generator to obtain a symmetric matrix.
This is unitarily equivalent to the anisotropic Heisenberg
model, known as the XXZ model, with the anisotropy parameter $\Delta=(q+q^{-1})/2$
with ``kink soliton'' boundary conditions.
For more information on the XXZ model, see \cite{PasquierSaleur,KomaNachtergaele}.
This relation has a long history.
See, for example, \cite{CaputoMartinelli}.

In Section \ref{sec:Outlook}, we will give a few more details of the relation between the XXZ model and the ASEP, in order to interpret
our main results in terms of the XXZ model.

Also, for the relation between the Mallows model and the ASEP, an important
reference point is to consider $q=1$ where one sees the relation between the uniform measure
on permutations and the SEP, where Liggett's stirring process gives a graphical representation,
which one may see in Chapter VIII of \cite{Liggett}.
(This also leads to duality, and in this regard, one may also refer for the $q\neq 1$ case to \cite{Giardina}.)

The most important points for us are 2: first the XXZ model is also frustration free \cite{GottsteinWerner}.
This is important, because that implies certain properties such as a spectral gap \cite{Nachtergaele}
and certain correlation structure in the ground state \cite{NachtergaeleSims}.
Secondly, there are certain equations related to the thermodynamic limits of the XXZ model
such as the Liouville PDE \cite{RibeiroKorepin}.
The Liouville PDE is known to have the boundary symmetry we mentioned in the last subsection,
related to the frustration free property \cite{LeviMartinaWinternitz}.
(Also, see Section 1.1 of the published version of Tao's blog, year 3, \cite{TaoY3},
for an elementary derivation of the symmetries and solution of the Liouville equation.)
In this paper we avoid the partial differential equations.
But that is an alternative route which has been explored before \cite{Starr}.

We do not try to relate the frustration free property of the XXZ model to the frustration free
property of the LDP optimization problem.
But we will state, in an appendix, the discrete version of the frustration free property
of the LDP.

\subsubsection{Outline for the rest of the paper}

A brief summary of our paper is this.
It is known that the ``pressure'' for the Mallows model is explicitly calculable.
More precisely, this is related to the $q$-Stirling's formula,
which is also related to the dilogarithm (although we will not discuss that).
Because of the symmetry we can reduce the 4-square problem to the calculation
of the pressure at  diluted inverse-temperatures.
This dilution, or ``temperature renormalization,'' arises in all mean-field
problems (where the statistical mechanics setup is  deficient because
the Hamiltonian itself explicitly depends on the system size).
The important point is that, due to the frustration free property, we may
reduce the four-square problem of calculating $I_{\beta}(W_{\theta_1,\theta_2}((t_{ij})_{i,j=1}^{2})$
to a 1-dimension optimization problem
related to the density of points in each subsquare, not the sub-permutations.
Even though the weak LDP is not strictly convex, this particular problem is.
That explains why there is a unique minimizer in this problem, despite the lack of convexity.

Although the problem is related to  interesting topics in quantum statistical mechanics,
the main results and tools, henceforth, will be purely probabilistic.
We do not make any further reference to quantum spin systems (such as the XXZ model)
or partial differential equations (such as the Liouville PDE).
The approach may be viewed as purely probabilistic by probabilists.

\section{Set-up}
\label{sec:Setup}

Let $\lambda$ be the standard Lebesgue measure on $[0,1]$.
Let $\lambda^{\otimes 2}$ denote the standard Lebesgue measure on $[0,1]^2$,
which is a Borel probability measure on $\R^2$.

For any $n \in \{2,3,\dots\}$ and $\beta \in \R$, let us define the measure
$\mu_{n,\beta} \in \mathcal{M}_{+,1}(([0,1]^2)^n)$ to be the absolutely continuous
measure with respect to $(\lambda^{\otimes 2})^{\otimes n}$ such that
$$
\frac{d \mu_{n,\beta}}{d(\lambda^{\otimes 2})^{\otimes n}}\big((x_1,y_1),
\dots,(x_n,y_n)\big)\,
=\, \frac{1}{Z_n(\beta)}\, \exp\left[-\beta H_n\big((x_1,y_1),\dots,(x_n,y_n)\big)\right]\, ,
$$
where
\begin{equation}
    \label{eq:HnDef} 
    H_n\big((x_1,y_1),\dots,(x_n,y_n)\big)\, 
    =\, \frac{1}{n-1}\, \sum_{i=1}^{n-1} \sum_{j=i+1}^{n} h\big((x_i,y_i),(x_j,y_j)\big)\, ,
\end{equation}
for 
\begin{equation}
\label{eq:hDef1}
h\big((x_i,y_i),(x_j,y_j)\big)\,
=\, \boldsymbol{1}_{(-\infty,0)}\big((x_i-x_j)(y_i-y_j)\big)\, ,
\end{equation}
and where $Z_n(\beta)$ is a normalization constant
$$
Z_n(\beta)\,
=\, \int_{([0,1]^2)^n} \exp\left[-\beta H_n\big((x_1,y_1),\dots,(x_n,y_n)\big)\right]\,
\prod_{i=1}^{n} d\lambda^{\otimes 2}(x_i,y_i)\, .
$$
The main result of this paper relates to the weak large deviation principle for this sequence of measures.
\begin{remark}
\label{rem:betaRescale}
We have chosen to exclude the sum of the diagonal terms $\frac{1}{2} \sum_{i=1}^{n} h((x_i,y_i),(x_i,y_i))$ in the definition of $H_n$ in (\ref{eq:HnDef}), and subsequently divided by $n-1$ instead of $n$.
This choice does not make any macroscopic difference in the large $n$ limit.
It is consistent with some conventions in some parts of mathematical physics and the study of statistical mechanics of spin systems.
Later, to match up with large deviation theory results, it will be preferable to normalize by $n$ instead of $n-1$.
This may be accomplished by a rescaling of the inverse temperature parameter $\beta \in \R$ by a factor $(1-n^{-1})$.
\end{remark}
Let us define the finite-volume approximation to the pressure
$$
p_n(\beta)\, =\, \frac{1}{n}\, \ln\left(Z_n(\beta)\right)\, .
$$

\subsection{Relation to the Mallows measure on permutations}

Given a parameter $q \in (0,\infty)$,
the Mallows  measure on permutations is a probability measure on the symmetric group.
More precisely, the probability mass function  is 
$P_{n,q} : S_n \to \R$,
where, for a given permutation $\pi  = (\pi_1,\dots,\pi_n) \in S_n$, we define the inversion number
and $P_{n,q}$ as
\begin{equation}
\label{eq:Pnq}
\operatorname{inv}(\pi)=\#\{(i,j)\,  :\, i<j \ \text { and }\ \pi_i>\pi_j\}\ 
\text{ and }\
P_{n,q}(\pi)\, =\, \frac{q^{\operatorname{inv}(\pi)}}{Z_{n,q}}\, , \text{ where }\
Z_{n,q}\, =\,  \sum_{\pi\in S_n} q^{\operatorname{inv}(\pi)}\, .
\end{equation}
Diaconis and Ram showed that the measure $P_{n,q}$ is related to the Iwahori-Hecke algebra \cite{DiaconisRam}.
One related fact is the special formula for the normalization:
\begin{equation}
\label{eq:DiaconisRam}
Z_{n,q}\,  =\, \prod_{k=1}^{n} \frac{1-q^k}{1-q}\, .
\end{equation}
The $q$-integers are defined as $[k]_q = (1-q^k)/(1-q)$ for each $k \in \N = \{1,2,\dots\}$.
The $q$-factorial function is 
$$
[n]_q!\, =\, \prod_{k=1}^{n} [k]_q\, .
$$
Hence $Z_{n,q} = [n]_q!$.

The relation between the measure $\mu_{n,\beta}$ in $\mathcal{M}_{+,1}(([0,1]^2)^n)$
and the probability mass function $P_{n,q} : S_n \to [0,1]$
is elucidated in the following lemma.
\begin{lemma}
For each $n \in \N$ and each $\beta \in \R$,
$$
p_n(\beta)\, =\, \frac{1}{n}\, \ln\left(\frac{[n]_q!}{n!}\right) \bigg|_{q=\exp(-\beta/(n-1))}\, .
$$
\end{lemma}
\begin{proof}
Since the Lebesgue measure is permutation invariant, we may symmetrize to calculate
$$
e^{np_n(\beta)}\, =\, \int_{([0,1]^2)^n} \frac{1}{n!}\, \sum_{\pi \in S_n} 
\exp\left[-\beta H_n\big((x_{\pi_1},y_1),\dots,(x_{\pi_n},y_n)\big)\right]\, \prod_{i=1}^{n} d\lambda(x_i)
\prod_{j=1}^{n} d\lambda(y_j)\, .
$$
For Lebesgue-a.e.\ choice of $(x_1,y_1),\dots,(x_n,y_n)$, we have
$$
\sum_{\pi \in S_n} 
e^{-\beta H_n\big((x_{\pi_1},y_1),\dots,(x_{\pi_n},y_n)\big)}\,
=\, 
\sum_{\pi \in S_n} 
\left(e^{-\beta/(n-1)}\right)^{\#\{(i,j)\, :\, i<j \text{ and } (x_{\pi_i}-x_{\pi_j})(y_i-y_j)<0\}}\, .
$$
But this quantity is precisely $Z_{n,q}$ for $q=\exp(-\beta/(n-1))$. So the result follows 
from (\ref{eq:DiaconisRam}).
\end{proof}
This proof demonstrates the relation between $\mu_{n,\beta}$ and $P_{n,q}$:
$$
\frac{d\mu_{n,\beta}}{d(\lambda^{\otimes 2})^{\otimes n}}\big((x_1,y_1),\dots,(x_n,y_n)\big)\,
=\, P_{n,q}(\pi)\, ,\ a.s.\, ,
$$
where $q = \exp(-\beta/(n-1))$ and $\pi = \pi\big((x_1,y_1),\dots,(x_n,y_n)\big)$
is the $(\lambda^{\otimes 2})^{\otimes n}$-almost surely unique permutation such that
$$
y_{\pi_i} \leq y_{\pi_j}\quad \Leftrightarrow\quad x_i \leq x_j\, .
$$

\begin{corollary}
\label{cor:pLimit}
For each $\beta \in \R$, the sequence $(p_n(\beta))_{n=1}^{\infty}$ converges to the limit
\begin{equation}
\label{eq:pFormula}
p(\beta)\,
=\, \int_0^1 \ln\left(\frac{1-e^{-\beta x}}{\beta x}\right)\, dx\, .
\end{equation}
\end{corollary}
\begin{proof}
A more general and precise result is true, which we will mention immediately after this proof.
This easy result follows from
$$
\frac{1}{n} \ln\left(\frac{[n]_q!}{n!}\right)_{q=\exp(-\beta/n)}\,
=\, \ln\left(\frac{1-e^{-\beta/n}}{\beta n}\right)-\frac{1}{\beta} \sum_{k=1}^{n} \ln\left(\frac{1-e^{-\beta x_{n,k}}}{\beta x_{n,k}}\right)\, \Delta x_n\, ,
$$
where
$$
x_{n,k}\, =\, \frac{k}{n}\ \text{ and }\ \Delta x_n\, =\, \frac{1}{n}\, .
$$
The first term is easily seen to converge to $0$.
The second term converges to the Riemann-Stieltjes integral.
Note that we rescaled $\beta$ by $(n-1)/n$ in this formula (for reasons as in Remark \ref{rem:betaRescale}). But the limiting formula (involving the integral)
is continuous in $\beta$. So this rescaling does not matter. 
\end{proof}
A more general and precise result than this one is true.
It is called the $q$-Stirling
formula.
It was first proved by Moak \cite{Moak}.
We will discuss this further in the Outlook, Section \ref{sec:Outlook},
since it is related to the quantitative version of our main result,
which may also be useful for the studying fluctuations, especially in the 
singular scaling of Bhatnagar and Peled from their paper \cite{BhatnagarPeled}.

\section{Statement of main results}

For two measures $\mu,\nu$ on a Borel probability
space on a compact metric space $\mathcal{X}$,
let us define the relative entropy in the usual way
$$
S(\mu\, |\, \nu)\, =\, \begin{cases} -\infty & \text{ if $\mu \not\ll \nu$,}\\
-\int_{\mathcal{X}} \ln\left(\frac{d\mu}{d\nu}(x)\right)\, d\mu(x) & \text{ if $\mu \ll \nu$.}
\end{cases}
$$
The way we have defined it, this is the negative of the Kullback-Leibler divergence.
Define the function $\mathcal{E} : \mathcal{M}_{+,1}([0,1]^2) \to \R$ as follows
\begin{equation}
\label{eq:calEdef}
\mathcal{E}(\nu)\, =\, \frac{1}{2} \int_{[0,1]^2} \int_{[0,1]^2} h\big((x_1,y_1),(x_2,y_2)\big)\, d\nu(x_1,y_1)\, d\nu(x_2,y_2)\, ,
\end{equation}
where $h$ is as defined before in (\ref{eq:hDef1}).
Then $\mathcal{E}(\nu)$ is the expectation of the energy in a product state $\nu \otimes \nu$.
\begin{lemma}
\label{lem:levelSets}
For each $\beta \in \R$, the function $\widetilde{\mathcal{I}}_{\beta} : \mathcal{M}_{+,1}([0,1]^2) \to \R \cup \{-\infty\}$ defined as 
\begin{equation}
\label{eq:ItildeDef}
\widetilde{\mathcal{I}}_{\beta}(\nu)\, =\, -S(\nu|\lambda^{\otimes 2})+\beta \mathcal{E}(\nu)
\end{equation}
has the following properties:
\begin{itemize}
\item[(a)] $\widetilde{\mathcal{I}}_{\beta}$ is lower semi-continuous on $\mathcal{M}_{+,1}([0,1]^2)$, and
\item[(b)] $\widetilde{\mathcal{I}}_{\beta}$ has compact level sets.
\end{itemize}
\end{lemma}
The properties are well-known for the relative entropy, which is the negative of the Kullback-Liebler divergence, relative to the uniform measure.
Therefore, at $\beta=0$, the lemma is trivial.

If $\mathcal{E}$ were continuous, then this conclusion would be immediate for all $\beta\neq 0$, as well.
But $h$ is discontinuous on the set of $((x_1,y_1),(x_2,y_2))$ in $[0,1]^2\times [0,1]^2$ where $(x_1-x_2)(y_1-y_2)=0$.
Therefore, $\mathcal{E}$ is not continuous.
However, we do know that $\mathcal{E}$ is bounded.
Moreover, one may
demonstrate that on the level sets of $\widetilde{\mathcal{I}}_0$, it is continuous.
That suffices, as one may show.
We will not prove this simple lemma.

Then the following is an important consequence of general principles.
\begin{proposition} 
\label{prop:Ellis}
Suppose $\beta \in \R$ is any fixed number.  For any closed subset $A \subseteq \mathcal{M}_{+,1}([0,1]^2)$,
\begin{multline*}
\limsup_{n \to\infty} \frac{1}{n} \ln\left[ \mu_{n,\beta(1-n^{-1})}\left(\left\{\big((x_1,y_1),\dots,(x_n,y_n)\big) \in ([0,1]^2)^n\, :\, \frac{1}{n} \sum_{k=1}^n \delta_{(x_k,y_k)} \in A\right\}\right)\right] \\
\leq\, \sup_{\nu \in A} \left(-\widetilde{\mathcal{I}}_{\beta}(\nu)- p(\beta)\right)\, ,
\end{multline*}
and for any open subset $A \subseteq \mathcal{M}_{+,1}([0,1]^2)$,
\begin{multline*}
\liminf_{n \to\infty} \frac{1}{n} \ln\left[ \mu_{n,\beta(1-n^{-1})}\left(\left\{\big((x_1,y_1),\dots,(x_n,y_n)\big) \in ([0,1]^2)^n\, :\, \frac{1}{n} \sum_{k=1}^n \delta_{(x_k,y_k)} \in A\right\}\right)\right] \\
\geq\, \sup_{\nu \in A} \left(-\widetilde{\mathcal{I}}_{\beta}(\nu)- p(\beta)\right)\, .
\end{multline*}
\end{proposition}
The rescaling of $\beta$ by the factor $(1-n^{-1})$ is related to the comment in Remark \ref{rem:betaRescale}.
There 
is also 
the following easy corollary. 
\begin{corollary}
\label{cor:Ellis}
Suppose $(\beta_n)_{n \in \N}$ is a real sequence such that $\lim_{n \to \infty} \beta_n = \beta$ for some $\beta \in \R$.
Then, for any closed subset $A \subseteq \mathcal{M}_{+,1}([0,1]^2)$,
\begin{multline*}
\limsup_{n \to\infty} \frac{1}{n} \ln\left[ \mu_{n,\beta_n}\left(\left\{\big((x_1,y_1),\dots,(x_n,y_n)\big) \in ([0,1]^2)^n\, :\, \frac{1}{n} \sum_{k=1}^n \delta_{(x_k,y_k)} \in A\right\}\right)\right] \\
\leq\, \sup_{\nu \in A} \left(-\widetilde{\mathcal{I}}_{\beta}(\nu)- p(\beta)\right)\, ,
\end{multline*}
and for any open subset $A \subseteq \mathcal{M}_{+,1}([0,1]^2)$,
\begin{multline*}
\liminf_{n \to\infty} \frac{1}{n} \ln\left[ \mu_{n,\beta_n}\left(\left\{\big((x_1,y_1),\dots,(x_n,y_n)\big) \in ([0,1]^2)^n\, :\, \frac{1}{n} \sum_{k=1}^n \delta_{(x_k,y_k)} \in A\right\}\right)\right] \\
\geq\, \sup_{\nu \in A} \left(-\widetilde{\mathcal{I}}_{\beta}(\nu)- p(\beta)\right)\, .
\end{multline*}
\end{corollary}
This corollary is easy using monotonicity and continuity of the large deviation rate functions with respect to the parameter $\beta$.
We will skip the proofs of the proposition and corollary.
(A complete detailed proof is available for an earlier version online.)
Using the corollary, we see {\em a posteriori} that the difference between using $\beta$ and $\beta (1-n^{-1})$, as was done in Proposition \ref{prop:Ellis}, does not matter in the 
$n \to \infty$ limit.

If we considered the discrete version of the Mallows measure, instead of the continuous case, then
this theorem would follow as a special case of a more general theorem proved by Mukherjee in \cite{Mukherjee}. 
We will discuss this more, momentarily.
Mukherjee also noted that the result had previously been proved by Trashorras \cite{Trashorras}.
\begin{remark}
\label{rem:ForRef1}
In the present paper we consider the continuous case of measures $\mu_{n,\beta}$, as we defined
them in (\ref{eq:H1}).
Mukherjee and Trashorras considered a discrete case, which is similar.
The difference is that in the continuous case 
 we consider the random empirical measure
$\mathfrak{m}^{(n)}_{(x_1,y_1),\dots,(x_n,y_n)}$ when the random point set $((x_1,y_1),\dots,(x_n,y_n))$ is $\mu_{n,\beta}$-distributed.
In the discrete setting one considers the Mallows measure $P_{n,q}$ on $S_n$ for $q$ satisfying
the relation $q=q(n)$ such that $\lim_{n \to \infty} n^{-1} (1-q(n))=\beta$, and then considers the 
random measure $n^{-1} \sum_{i=1}^{n} \delta_{(i/n,\pi_i/n)}$ as the analogue of the random empirical measure, where $\pi \in S_n$ is $P_{n,q}$-distributed.
\end{remark}

\begin{remark}
\label{rem:ForRef1b}
A referee has pointed out that in the discrete case considered in \cite{Mukherjee,Trashorras} the large deviation rate function is changed in a simple way:
it becomes infinite except for measures
satisfying that both marginals are uniform; for measures where both marginals are uniform it is as above.
That is because in a permutation $\pi \in S_n$, the marginals of $n^{-1} \sum_{i=1}^{n} \delta_{(i/n,\pi_i/n)}$ are uniform on $\{1/n,\dots,n/n\}$.
\end{remark}

For this result, Proposition \ref{prop:Ellis}, one may refer to a beautiful monograph \cite{Ellis}.
Ellis considered large deviation principles for mean-field statistical mechanical models.
Then this proposition follows almost exactly from Theorems II.7.1 and II.7.2 on pages 51-52 of Ellis's monograph,
which is the section on Varadhan's lemma.

The technical issue is that to apply Varadhan's lemma, one usually assumes that the perturbing function
(which we would call the Hamiltonian) is bounded and a continuous function with respect to the topology used to state the large
deviation principle for the {\em a priori} sequence of measures.
But in the present paper the Hamiltonian is not continuous.
This is a frequent issue. For example, it complicates the proof of the first large deviation result for random
matrices in \cite{AndersonGuionnetZeitouni} (which is Theorem 2.6.1 in the section on large deviations starting on page 70).
For us, the Hamiltonian is bounded.
One can proceed by showing that it is well-approximated by continuous functions.

A referee has pointed out that in a standard reference by Dembo and Zeitouni this has been carried out
for the contraction principle as opposed to Varadhan's lemma \cite{DemboZeitouni}. (The two results are complementary pillars
of large deviation theory.)
In this section 4.2 (starting on page 126) the approximations would be called exponential approximations.
One may split any argument into two pieces: first proving that one has exponential approximations, and then
establishing the desired result (Varadhan's lemma in this case) follows from that.
This is useful, and if any reader wants complete details they may consult the old version of the article on the arXiv.

The large deviation rate function is
\begin{equation}
\label{eq:I}
\mathcal{I}_{\beta}(\nu)\, =\, \widetilde{\mathcal{I}}_{\beta}(\nu)+ p(\beta)\,
=\, -\left(S(\nu|\lambda^{\otimes 2}) - \beta \mathcal{E}(\nu) - p(\beta)\right)\, .
\end{equation}
Part of Ellis's approach to this proposition entails the fact that $\mathcal{I}_{\beta}(\nu)$ has 
infimum
equal to $0$, as is needed on basic probabilistic grounds.
The large deviation rate function $\mathcal{I}_{\beta}$ is lower-semi continuous.
Therefore, any {\em infimizing sequence} possesses a limit point which is a minimizer.
In particular, there is at least one minimizer.
Sometimes we will denote this existential minimizer as
$\widetilde{\nu}^*_{\beta} \in \mathcal{M}([0,1]^2)$.
In other words, $\widetilde{\nu}^*_{\beta}$ will be a minimizer, which we know must exist
by these general principles.
But the notation is not meant to imply that there are no other minimizers, which is not immediately guaranteed
by general principles, at least not at the outset.
(After we establish uniqueness, we will change notation to $\nu^*_{\beta}$.)

\begin{remark}
\label{rem:ForRef3}
The nature of our short paper is two-fold. First, we establish uniqueness.
Second, in the process of doing that, we discover a  natural way of determining the explicit
formula for the unique measure which we will later call $\nu^*_{\beta}$.
The approach is to apply standard calculus methods as in variational analysis.
In the course of the arguments, the explicit formula will appear.
\end{remark}

At various points of the argument, one could cross-check with the known formula as 
an exercise. 
This is an effective approach to gauging the proof.
(This has been pointed out by a referee.)
We will not interrupt the proof to do this.

\smallskip
\noindent
\underline{\em Statement of main result:}
Our main result is uniqueness of the minimizer.
We state this in a sequence of steps.
The main step involves the ``four-square problem,''
which we describe, shortly.
First we would like to comment, briefly, on the papers of Mukherjee and Trashorras, which
are critical for our own short article.

\subsection{Short discussion}
Let us quickly comment on one important aspect of Mukherjee's paper \cite{Mukherjee}, first.
He not only calculated the large deviation principle for the Mallows model, but also 
for many other non-uniform measures on $S_n$, of which the Mallows measure is just one.
Mallows, himself, considered various measures on $S_n$, and what we are calling the ``Mallows measure,''
here, is actually just the Mallows measure relative to a particular distance function, which is the minimum number
of nearest neighbor transpositions (i.e., transpositions of the form $(i,i+1)$ for some $i \in \{1,\dots,n-1\}$, sometimes
called Coxeter generators) needed to transform a given permutation into the identity permutation.
Mukherjee has explained all of this.
Trashorras has also considered many generalizations of the uniform measure on permutations, too.
In this sense, they are the same.
But Mukherjee also considered properties of the large deviation problems.

We would also like to mention that Trashorras was motivated by models of permutations arising from Adams, Bru, Dorlas and K\"onig
relating to Bose condensation \cite{ABK,ABK2,AD,AK}. Another important reference is by Betz and Ueltschi \cite{BetzUeltschi}.
Following a methodology of Ueltschi, they are also related to quantum spin systems \cite{GoldschmidtUeltschiWindridge}.
But we have not found a direct connection to the Mallows model, yet, since the cycle type is more important than the inversion
number for all of those applications. Nevertheless, there is a direct relation between the Mallows measure and the XXZ model
that we mentioned in the introduction. We will return to this in the appendix.

There is a particular type of statistic, which one may call ``linear statistics'' for which
Mukherjee proved uniqueness of optimizers of the LDP, in those cases.
He considered measures $Q_{n,\theta}(\pi) = \exp[\theta\sum_{i=1}^{n} f(i/n,\pi_i/n) - \ln(Z_n(f,\theta))]$
for a given continuous function $f : [0,1]^2 \to \R$.
In this case, the large deviation rate function is equal to the Kullback-Leibler divergence (which is the negative of the relative entropy, relative to the uniform measure)
plus a linear functional of the measure.
But the Kullback-Leibler divergence is strictly convex, and addition of a bounded linear functional does not affect this.
One might expect that if there is uniqueness of the optimizer of the LDP for the Mallows measure that this follows from convexity.
But, the rate functional $\mathcal{I}_{\beta}$ is not strictly convex on $\mathcal{M}([0,1]^2)$. 
(Note that if we describe everything other than the Kullback-Leibler divergence as an effective Hamiltonian, then this is quadratic
in the measure, somewhat like the logarithmic potential although less singular, not linear. Hence it need not preserve convexity.)
We will give an example calculation to show this in the appendix.
Nevertheless, there is a class of events that does lead to convexity for a 1-parameter family of sub-problems.
That is how we proceed.
This is the 4-square problem, which we now describe.

\subsection{Four square problem}
For any $\theta \in (0,1)$, define subsets of $[0,1]$ as
$$
L_1(\theta)\, =\, [0,\theta]\, ,\qquad
L_2(\theta)\, =\, (\theta,1]\, .
$$
For each point $(\theta_1,\theta_2) \in (0,1)^2$,
we define four rectangular subsets of $[0,1]^2$:
\begin{equation}
\label{eq:defLambbda}
\Lambda_{ij}(\theta_1,\theta_2)\, =\, L_i(\theta_1) \times L_j(\theta_2)\, ,
\text{ for $i,j \in \{1,2\}$.}
\end{equation}
Let $\Sigma_4 = \{(t_{11},t_{12},t_{21},t_{22}) \in [0,1]^4\, :\, t_{11}+t_{12}+t_{21}+t_{22}=1\}$.
Given a point in $M$,
we define a Borel subset of $\mathcal{M}_{+,1}([0,1]^2)$, as
$$
W_{\theta_1,\theta_2}(t_{11},t_{12},t_{21},t_{22})\,
=\, 
\left\{\nu \in \mathcal{M}_{+,1}([0,1]^2)\, :\,
\forall i,j \in \{1,2\}\, ,\
\nu(\Lambda_{ij}(\theta_1,\theta_2)\, =\, t_{ij}
\right\}
$$
What we call the ``four-square problem'' is to calculate the large deviation of this set.
Solving this problem for all possible values $\theta_1,\theta_2 \in (0,1)^2$
and all $(t_{11},t_{12},t_{21},t_{22}) \in \Sigma_4$ will lead to all
the minimizers of $\mathcal{I}_{\beta}$.
\begin{theorem}
\label{thm:4square}
For each $\beta \in \R$,
and each $(\theta_1,\theta_2) \in (0,1)^2$, and each $(t_{11},t_{12},t_{21},t_{22}) \in \Sigma_4$,
$$
\min\{ \mathcal{I}_{\beta}(\nu)\, :\,
\nu \in W_{\theta_1,\theta_2}(t_{11},t_{12},t_{21},t_{22})\}\,
=\, \widetilde{\Phi}_{\beta}(\theta_1,\theta_2;t_{11},t_{12},t_{21},t_{22})\, ,
$$
where
\begin{equation}
\label{eq:4squareC}
\begin{split}
\widetilde{\Phi}_{\beta}(\theta_1,\theta_2;t_{11},t_{12},t_{21},t_{22})\, 
&\stackrel{\mathrm{def}}{:=}\, p(\beta) 
+ \sum_{i,j=1}^{2} t_{ij} \ln\left(\frac{t_{ij}}{|\Lambda_{ij}|}\right) 
+
\sum_{i,j=1}^{2} t_{ij} p(\beta t_{ij}) \\
&\qquad - (t_{11}+t_{12}) p(\beta (t_{11}+t_{12})) - (t_{11}+t_{21}) p(\beta (t_{11}+t_{21}))\\
&\qquad
- (t_{12}+t_{22}) p(\beta (t_{12}+t_{22})) - (t_{21}+t_{22}) p(\beta (t_{21}+t_{22}))\\
&\qquad + \beta t_{12} t_{21}\, .
\end{split}
\end{equation}
\end{theorem}
Actually a subset of the four-square problems suffices to find the minimizers of $\mathcal{I}_{\beta}$.
But the formula for the optimum in general is an additional result from the present analysis.
Moreover, one may obtain a formula for the unique optimizer, which we will describe later.

\begin{remark}
We thank an anonymous referee for pointing out the utility in explicating the optimizers, themselves, which 
we will do in Subsection \ref{subsec:demo} and Section \ref{sec:Outlook}.
\end{remark}

Given a measure $\nu \in \mathcal{M}_{+,1}([0,1]^2)$, 
let us define the $x$ and $y$ marginals as $\nu_X,\nu_Y \in \mathcal{M}_{+,1}([0,1])$ defined as
\begin{equation}
\label{eq:marginals}
\nu_X(\cdot)\, =\, \nu(\cdot \cap [0,1])\
\text{ and }\
\nu_Y(\cdot)\, =\, \nu([0,1]\cap \cdot)\, .
\end{equation}
\begin{theorem}
\label{thm:Lebesgue}
For each $\beta \in \R$, any 
$\nu \in \mathcal{M}_{+,1}([0,1]^2)$ satisfying
$\mathcal{I}_{\beta}(\nu)  = 0$,
$\nu_X = \nu_Y = \lambda$.
\end{theorem}
Because of the lemma, we may restrict to $(t_{11},t_{12},t_{21},t_{22}) \in \Sigma_4$
such that $t_{11}+t_{12} = \theta_1$ and $t_{11}+t_{21}=\theta_2$
because those are the $\lambda$ measures of $[0,\theta_1]$ and $[0,\theta_2]$.
This reduces the parameter from general $(t_{11},t_{12},t_{21},t_{22}) \in \Sigma_4$
to just $t_{11}$ in a certain interval.
We define 
$$
I_{\theta_1,\theta_2}\,
=\,
[\max\{0,\theta_1+\theta_2-1\},\min\{\theta_1,\theta_2\}]\, .
$$
Then we define $\Phi_{\beta}(\theta_1,\theta_2;\cdot) : I_{\theta_1,\theta_2} \to \R$ as 
$$
\Phi_{\beta}(\theta_1,\theta_2;t)\, =\, \widetilde{\Phi}_{\beta}(\theta_1,\theta_2;t,\theta_1-t,\theta_2-t,1-\theta_1-\theta_2+t)\, .
$$
\begin{theorem}
\label{thm:main}
For each $\beta \in \R$,
and each $(\theta_1,\theta_2) \in (0,1)^2$,
the function $\Phi_{\beta}(\theta_1,\theta_2;\cdot) : I_{\theta_1,\theta_2} \to \R$ is strictly convex,
and the unique critical point is given by 
\begin{equation}
\label{eq:Rdefin}
t\, =\, R_{\beta}(\theta_1,\theta_2)\, \stackrel{\mathrm{def}}{:=}\,
-\frac{1}{\beta}\, \ln\left(1 - \frac{(1-e^{-\beta \theta_1})(1-e^{-\beta \theta_2})}{1-e^{-\beta}}\right)\, .
\end{equation}
\end{theorem}
This theorem easily leads to the following, which is the main summary of the results.
\begin{theorem}
\label{thm:summary}
The unique measure $\nu \in \mathcal{M}_{+,1}([0,1]^2)$ such that $\mathcal{I}_{\beta}(\nu)=0$ is $\nu=\nu_{\beta}^*$, where $d\nu_{\beta}^*(x,y) = \rho_{\beta}^*(x,y)\, dx\, dy$ for
$x,y \in [0,1]^2$, where 
$$
\rho_{\beta}(x,y)\, =\, \frac{\partial^2}{\partial x\, \partial y}\, 
R_{\beta}(x,y)\, .
$$
\end{theorem}
\begin{proof}
Any measure $\nu \in \mathcal{M}_{+,1}([0,1]^2)$ such that $\mathcal{I}_{\beta}(\nu)=0$ must satisfy $\nu_X=\nu_Y=\lambda$ by Theorem \ref{thm:Lebesgue}.
Then, for any $(\theta_1,\theta_2) \in \mathcal{M}_{+,1}([0,1]^2)$ the measure $\nu$ is in $W_{\theta_1,\theta_2}(t,\theta_1-t,\theta_2-t,1-\theta_1-\theta_2+t)$
for $t = \nu(\Lambda_{11}(\theta_1,\theta_2))$. So by Theorem \ref{thm:4square}, 
$$
\Phi_{\beta}(\theta_1,\theta_2;t)=\widetilde{\Phi}_{\beta}(\theta_1,\theta_2;t,\theta_1-t,\theta_2-t,1-\theta_1-\theta_2+t)\leq \mathcal{I}_{\beta}(\nu)=0\, .
$$
So that means $\widetilde{\Phi}_{\beta}(\theta_1,\theta_2;t,\theta_1-t,\theta_2-t,1-\theta_1-\theta_2+t)=\mathcal{I}_{\beta}(\nu)=0$ because $0$ is the minimum of
$\mathcal{I}_{\beta}$ (so a lower bound for every $\widetilde{\Phi}_{\beta}(\theta_1,\theta_2;t_1,t_2,t_3,t_4)$).
So $t=\nu(\Lambda_{11}(\theta_1,\theta_2))$ is the unique critical point of $\Phi_{\beta}(\theta_1,\theta_2;t)$.

But then, by Theorem \ref{thm:main}, we see that this means that $t=R_{\beta}(\theta_1,\theta_2)$. In other words, $\nu([0,\theta_1]\times [0,\theta_2]) = R_{\beta}(\theta_1,\theta_2)$.
These rectangular measures completely characterize $\nu$: in fact they give the standard formula for the multidimensional distribution function. So uniqueness is proved.
We now call it $\nu_{\beta}^*$.
Note that $\nu_{\beta}^* \ll \lambda^{\otimes 2}$ because the relative entropy is not $-\infty$.
Therefore, there is a density, which may be calculated by differentiating the distribution function.
\end{proof}
Let us comment on the relation to the absence of phase transitions for this model.

For a mean-field model from statistical mechanics, one needs to replace the usual
Dobrushin-Lanford-Ruelle definition of equilibrium states by an appropriate analogue.
When a weak large deviation principle exists, as it does here, the correct analogue,
in the sense of the Boltzmann-Gibbs variational principle, may be viewed as $\mathcal{I}_{\beta}(\nu)=0$.
Therefore, this result is a version of absence of phase transition in this model,
in the sense of uniqueness of equilibrium states.

\section{Standardizing the measure}

This is a solvable model.
This is manifest in the symmetry of $\mathcal{I}_{\cdot}(\cdot)$.
We describe this, now.

\begin{definition}
\label{def:frakN}
Suppose that $\nu^{(0)} \in \mathcal{M}_{+,1}([0,1]^2)$ satisfies $\nu^{(0)} \ll \lambda^{\otimes 2}$.
Suppose that $G_X,G_Y : [0,1] \to [0,1]$ are two absolutely continuous probability distribution functions.
Then we may define a new measure $\nu = \mathfrak{N}(\nu^{(0)},G_X,G_Y)$ as 
\begin{equation}
\label{eq:scaling}
\nu([0,x]\times [0,y])\, =\, \nu^{(0)}([0,G_X(x)]\times [0,G_Y(y)])\, .
\end{equation}
\end{definition}

\begin{prop}
\label{prop:scaling}
Suppose that $\nu^{(0)} \in \mathcal{M}_{+,1}([0,1]^2)$ satisfies $\nu^{(0)} \ll \lambda^{\otimes 2}$ and $\nu^{(0)}_X=\nu^{(0)}_Y=\lambda$.
Suppose that $G_X,G_Y : [0,1] \to [0,1]$ are two absolutely continuous probability distribution functions.
Defining $\nu = \mathfrak{N}(\nu^{(0)},G_X,G_Y)$, we have that $\nu_X([0,a]) = G_X(a)$, $\nu_Y([0,a]) = G_Y(a)$ for all $a \in [0,1]$.
Moreover,  
\begin{equation}
\label{eq:entropyFact0}
S(\nu\, |\, \lambda^{\otimes 2})\,
=\, S(\nu^{(0)}\, |\, \lambda^{\otimes 2}) +
S(\nu_X\, |\, \lambda^{\otimes 1}) + S(\nu_Y\, |\, \lambda^{\otimes 1})\, ,
\end{equation}
and 
\begin{equation}
\label{eq:energyFact0}
\mathcal{E}(\nu)\, =\, \mathcal{E}(\nu^{(0)})\, .
\end{equation}
\end{prop}
\begin{proof}
The fact about the marginals follows directly from the definition (\ref{eq:scaling}) and the definition of the marginals (\ref{eq:marginals}) and $\nu^{(0)}_X = \nu^{(0)}_Y = \lambda$.

Let us write $\rho^{(0)} : [0,1]^2 \to [0,\infty)$ for the density associated to $\nu^{(0)}$. Note that $\int_{0}^{1} \rho^{(0)}(x,y)\, dy= 1$ for $x$, $\lambda$-a.e., 
and similarly $\int_{0}^{1} \rho^{(0)}(x,y)\, dx = 1$ for $y$, $\lambda$-a.e, because $\nu^{(0)}_X = \nu^{(0)}_Y=\lambda$.
Let us write $g_X:[0,1] \to [0,\infty)$ and $g_Y:[0,1] \to [0,\infty)$ for the density functions associated to $G_X$ and $G_Y$.
Then, from (\ref{eq:scaling}) and the chain rule we have
\begin{equation}
\label{eq:densityScaling}
\frac{d\nu}{d\lambda^{\otimes 2}}(x,y)\, =\, \rho^{(0)}(G_X(x),G_Y(y)) g_X(x) g_Y(y)\, .
\end{equation}
Therefore,
\begin{align*}
S(\nu\, |\, \lambda^{\otimes 2})\,
&=\, - \int_{0}^{1} \int_{0}^{1} \ln\left(\rho^{(0)}(G_X(x),G_Y(y))\right) \rho^{(0)}(G_X(x),G_Y(y)) g_X(x) g_Y(y)\, dx\, dy\\
&\qquad -\int_{[0,1]^2} \ln\left(g_X(x)\right)\, d\nu(x,y)
 -\int_{[0,1]^2} \ln\left(g_Y(x)\right)\, d\nu(x,y)\, .
\end{align*}
In the second integral we integrate over $y$ first and use the fact that the marginal $\nu_X$ has density function $g_X$ and in the third integral we integrate over $x$ first and use the 
fact that $\nu_Y$ has density function $g_Y$.
In the first integral we make the change of variables $x = G_X^{I}(u)$ and $y=G_Y^I(v)$, and then we obtain 
\begin{align*}
S(\nu\, |\, \lambda^{\otimes 2})\,
&=\, - \int_{0}^{1} \int_{0}^{1} \ln\left(\rho^{(0)}(x,y)\right) \rho^{(0)}(x,y)\, dx\, dy\\
&\qquad -\int_{0}^{1} \ln\left(f_X(x)\right) f_X(x)\, dx
-\int_{0}^{1} \ln\left(f_Y(y)\right) f_Y(y)\,  dy\, .
\end{align*}
The first integral follows from the chain rule for Stieltjes integrals.
See, for example, \cite{RieszNagy}, Chapter III, especially Section 58.

The result for the energy follows from a similar calculation, using the same method.
\end{proof}

Now, given any $\nu \in \mathcal{M}_{+,1}([0,1]^2)$,
if $\nu \ll \lambda^{\otimes 2}$, then $\nu_X$ and $\nu_Y$, defined in (\ref{eq:marginals})
are both absolutely continuous with respect to $\lambda$.
We define $F_{\nu,X},F_{\nu,Y} : [0,1] \to [0,1]$ as 
$$
F_{\nu,X}(a)\, =\, \nu_X([0,a])\
\text{ and }\
F_{\nu,Y}(a)\, =\, \nu_Y([0,a])\, ,
$$
for each $a \in [0,1]$. These functions are both continuous.
\begin{definition}
Given a distribution function $G : [0,1] \to [0,1]$, let us denote the generalized inverse as $G^I : [0,1] \to [0,1]$,
where the condition to be a generalized inverse is
$$
\forall x \in [0,1]\, ,\
G^I(x)\, =\, \inf(\mathcal{U}_G(x))\, ,\ \text{ where }\ \mathcal{U}_{G}(x)\, =\, \{a \in [0,1]\, :\, G(a)\geq x\}\, .
$$
\end{definition}
Generally speaking, by right-continuity of $G$, we have $G(G^I(x)) = \inf\{G(a)\, \, :\, a \in \mathcal{U}_G(x)\}\, \geq\, x$.
Also, if $y<G^I(x)$ then $y \not\in \mathcal{U}_G(x)$ so $G(y)<x$.
Hence,
\begin{equation}
\label{eq:GIconditions}
G(G^I(x)) \geq x\, ,\quad \text{and}\quad
\Big(y<G^I(x)\ \Rightarrow\ G(y) < x\Big)\, .
\end{equation}
That is true for any distribution function. More is true if $G$ is continuous.
\begin{lemma}
\label{lem:GI}
Suppose that $G : [0,1] \to [0,1]$ is a continuous distribution function.
Then $G^I$ is a right inverse for $G$: for all $x \in [0,1]$, $G(G^I(x))=x$.
\end{lemma}
\begin{proof}
Suppose that we had $G(G^I(x))>x$. Since $G$ is continuous, there would exist some $y<G^I(x)$ such that $G(y)>x$, as well. But this contradicts (\ref{eq:GIconditions}).
\end{proof}
\begin{definition}
If $\nu \in \mathcal{M}_{+,1}([0,1]^2)$ satisfies $\nu \ll \lambda^{\otimes 2}$, then define
$\widehat{\nu} \in \mathcal{M}_{+,1}([0,1]^2)$ to be the measure such that
\begin{equation}
\label{eq:widehatnuDef}
\widehat{\nu}([0,x] \times [0,y])\,
=\, \nu(\left[0,F_{\nu,X}^I(x)\right] \times \left[0,F_{\nu,Y}^I(y)\right])\, ,
\end{equation}
with the notation as above.
\end{definition}
The important relation between $\nu$ and $\widehat{\nu}$ is  the following.
\begin{lemma}
\label{lem:normalized}
If $\nu \in \mathcal{M}_{+,1}([0,1]^2)$ satisfies $\nu \ll \lambda^{\otimes 2}$, then $(\widehat{\nu})_X=(\widehat{\nu})_Y=\lambda$
and 
\begin{equation}
\label{eq:normalized}
\nu([0,x]\times [0,y])\, =\, \widehat{\nu}([0,F_{\nu,X}(x)]\times [0,F_{\nu,Y}(y)])\, ,
\end{equation}
for all $x,y \in [0,1]$.
In other words, $\nu = \mathfrak{N}(\widehat{\nu},F_{\nu,X},F_{\nu,Y})$.
\end{lemma}
\begin{proof}
To simplify notation, let us just write $F_X,F_Y,F_X^I,F_Y^I$ in place of $F_{\nu,X},F_{\nu,Y},F_{\nu,X}^I,F_{\nu,Y}^I$.

By definition 
$$
(\widehat{\nu})_X([0,x])\, =\, \widehat{\nu}([0,x] \times [0,1])\, =\, \nu([0,F^I_X(x)] \times [0,F^I_Y(1)])\, .
$$
Now we note that this means
$$
(\widehat{\nu})_X([0,x])\, \leq\, \nu([0,F_X^I(x)]\times [0,1])\, =\, \nu_X([0,F_X^I(x)])\, =\, F_X(F_X^I(x))\, =\, x\, ,
$$
using Lemma \ref{lem:GI}. But in fact,
$$
x - (\widehat{\nu})_X([0,x])\, =\, \nu([0,F_X^I(x)]\times (F_Y^I(1),1])\, \leq\, \nu([0,1]\times  (F_Y^I(1),1])\,
=\, \nu_Y((F_Y^I(1),1])\, ,
$$
and this equals $F_Y(1) - F_Y(F_Y^I(1))$, i.e., $1-F_Y(F_Y^I(1))$, which is $0$, again by Lemma \ref{lem:GI}.
So this proves that $(\widehat{\nu})_X = \lambda$, and the corresponding fact for the $y$ marginal follows, similarly.

Now, applying (\ref{eq:widehatnuDef}) with $x$ replaced by $F_X(x)$ and $y$
replaced by $F_Y(y)$,
\begin{equation}
\label{eq:widehatFirst}
\widehat{\nu}([0,F_X(x)] \times [0,F_Y(y)])\,
=\, \nu(\left[0,F_X^I(F_X(x))\right] \times \left[0,F_Y^I(F_Y(y))\right])\, .
\end{equation}
We note that (\ref{eq:GIconditions}) implies that if $a<F_X^I(F_X(x))$ then $F_X(a)<F_X(x)$.
Since $F_X$ is non-decreasing this implies that for every $a \in [0,1]$, if we have $a<F_X^I(F_X(x))$
then $a<x$. In other words, this all implies that $F_X^I(F_X(x)) \leq x$.
Using this, we may see that (\ref{eq:widehatFirst}) implies that
$$
\widehat{\nu}([0,F_X(x)] \times [0,F_Y(y)])\, \leq\, \nu([0,x]\times [0,y])\, .
$$
But then by the same type of argument as before, using (\ref{eq:widehatFirst}), we actually have
\begin{multline*}
\nu([0,x]\times [0,y]) - \widehat{\nu}([0,F_X(x)] \times [0,F_Y(y)])\\ 
\leq\, \nu_X([0,x]) - \nu_X([0,F_X^I(F_X(x))]) + 
\nu_Y([0,y]) - \nu_Y([0,F_Y^I(F_Y(y))]) \, .
\end{multline*}
And this equals $0$.
\end{proof}
Finally, we state the scaling properties of the entropy and energy, under the transformation of $\nu$ to $\widehat{\nu}$.
\begin{cor}
\label{cor:normalized}
For any $\nu \in \mathcal{M}_{+,1}([0,1]^2)$ with $\nu \ll \lambda^{\otimes 2}$,
\begin{equation}
\label{eq:entropyFact}
S(\nu\, |\, \lambda^{\otimes 2})\,
=\, S(\widehat{\nu}\, |\, \lambda^{\otimes 2}) +
S(\nu_X\, |\, \lambda^{\otimes 1}) + S(\nu_Y\, |\, \lambda^{\otimes 1})\, ,
\end{equation}
and 
\begin{equation}
\label{eq:energyFact}
\mathcal{E}(\nu)\, =\, \mathcal{E}(\widehat{\nu})\, .
\end{equation}
\end{cor}
\begin{proof}
Combine Proposition \ref{prop:scaling} and Lemma \ref{lem:normalized}.
\end{proof}
We may now prove Theorem \ref{thm:Lebesgue}. (We will prove Theorem \ref{thm:4square} in a later section.)

\begin{proofof}
{\bf Proof of Theorem \ref{thm:Lebesgue}:}
From (\ref{eq:entropyFact}) and (\ref{eq:energyFact}), it follows that
\begin{equation}
\label{eq:Ihat}
\mathcal{I}_{\beta}(\nu)\, =\, \mathcal{I}_{\beta}(\widehat{\nu}) - S(\nu_X|\lambda)
-S(\nu_Y|\lambda)\, .
\end{equation}
Hence 
$$
\mathcal{I}_{\beta}(\widehat{\nu})\, =\,
\mathcal{I}_{\beta}(\nu) 
+ S(\nu_X|\lambda)
+ S(\nu_Y|\lambda)\, .
$$
But the relative entropy is nonpositive and equals $0$ only at the unique maximizer $\lambda$.
Therefore, since $\mathcal{I}_{\beta}$ can never be negative, we see that
$\mathcal{I}_{\beta}(\nu)$ can equal $0$ only if $\nu_X=\nu_Y=\lambda$.
\end{proofof}
We can think of this proof as a partial solution of the one-square problem, for $[0,1]^2$.
More specifically, while it does not give the unique measure $\nu \in \mathcal{M}([0,1]^2)$
solving $\mathcal{I}_{\beta}(\nu)=0$, it does give the 
$x$ and $y$ marginals.
The next step is to consider the two-square problem.
For the two-square problem, we will not determine the exact marginals.
But we will prove a scaling which will ultimately let us solve the four-square problem.

\section{The two-square problem}
\label{sec:2square}

Given any $\theta \in (0,1)$ we define $\Lambda_1(\theta) = [0,1] \times [0,\theta]$
and $\Lambda_2(\theta) = [0,1] \times (\theta,1]$.
Since $\nu_Y=\lambda$,
we have $\nu(\Lambda_1)=\theta$
and $\nu(\Lambda_2) = 1-\theta$.
We define two new measures $\nu^{(1)},\nu^{(2)} \in \mathcal{M}_{+,1}([0,1]^2)$ by
$$
\nu^{(1)}(\cdot)\, =\, \theta^{-1} \nu(\cdot \cap \Lambda_1)\ \text{ and }\
\nu^{(2)}(\cdot)\, =\, (1-\theta)^{-1} \nu(\cdot \cap \Lambda_2)\, .
$$
Since $\nu_Y = \lambda$, we have for the $y$-marginals of $\nu^{(1)}$ and $\nu^{(2)}$:
\begin{equation}
\label{eq:yMar}
\nu^{(1)}_Y(\cdot )\, =\, \theta^{-1} \lambda(\cdot \cap [0,\theta])\
\text{ and }\
d\nu^{(2)}_Y(y)\, =\, (1-\theta)^{-1} \lambda(\cdot \cap (\theta,1])\, .
\end{equation}
The $x$ marginals $\nu^{(1)}_{X}$ and $\nu^{(2)}_X$
may be more complicated. But we can prove the following result.
\begin{lemma}
\label{lem:2square}
Assuming that $\nu \in \mathcal{M}_{+,1}([0,1]^2)$ has $\nu \ll \lambda^{\otimes 2}$
and $\nu_Y = \lambda$,
using the notation above,
\begin{align*}
\mathcal{I}_{\beta}(\nu)\,
&=\, p(\beta) - \theta p(\beta \theta) - (1-\theta) p(\beta(1-\theta)) 
 +\theta \mathcal{I}_{\beta \theta}(\widehat{\nu}^{(1)})
+(1-\theta)\mathcal{I}_{\beta (1-\theta)}(\widehat{\nu}^{(2)}) \\
&\qquad - \theta S(\nu_X^{(1)}\, |\, \lambda)
-(1-\theta) S(\nu_X^{(2)}\, |\, \lambda)
+\beta \theta(1-\theta) \int_{0}^{1} \int_{0}^{1}
\mathbf{1}_{(-\infty,0)}(x_2-x_1)\, d\nu_X^{(1)}(x_1)\, d\nu_X^{(2)}(x_2)\, .
\end{align*}
\end{lemma}
\begin{proof}
A straightforward computation using the definition of the entropy shows that, defining $s(\theta,1-\theta) = -\theta \ln(\theta) - (1-\theta) \ln(1-\theta)$,
\begin{align*}
\mathcal{I}_{\beta}(\nu)\,
&=\, p(\beta)-s(\theta,1-\theta) - \theta S(\nu^{(1)} | \lambda^{\otimes 2})
- (1-\theta) S(\nu^{(2)} | \lambda^{\otimes 2})
+ \beta \theta^2 \mathcal{E}(\nu^{(1)})
+ \beta (1-\theta)^2 \mathcal{E}(\nu^{(2)})\\
&\qquad 
+ \beta \theta (1-\theta) \int_{[0,1]^2} \int_{[0,1]^2} h\big((x_1,y_1),(x_2,y_2)\big)\, d\nu^{(1)}(x_1,y_1)\, d\nu^{(2)}(x_2,y_2)\, .
\end{align*}
We note that, because $\nu^{(1)}$ has support in $\Lambda^{(1)}$ and  $\nu^{(2)}$ has support in $\Lambda_2$
 we may rewrite this as 
\begin{align*}
\mathcal{I}_{\beta}(\nu)\,
&=\, p(\beta)-s(\theta,1-\theta) - \theta S(\nu^{(1)} | \lambda^{\otimes 2})
- (1-\theta) S(\nu^{(2)} | \lambda^{\otimes 2})
+ \beta \theta^2 \mathcal{E}(\nu^{(1)})
+ \beta (1-\theta)^2 \mathcal{E}(\nu^{(2)})\\
&\qquad 
+ \beta \theta (1-\theta) \int_{0}^1 \int_{0}^1 \mathbf{1}_{(-\infty,0)}(x_2-x_1)\, d\nu^{(1)}_X(x_1)\, d\nu^{(2)}_X(x_2)\, .
\end{align*}
Now we note that using the definition of $\mathcal{I}_{\beta}$ for all $\beta$'s, this may be rewritten as 
\begin{align*}
\mathcal{I}_{\beta}(\nu)\,
&=\, p(\beta)-s(\theta,1-\theta) +\theta \mathcal{I}_{\beta \theta}(\nu^{(1)}) - \theta p(\beta \theta)
+ (1-\theta) \mathcal{I}_{\beta (1-\theta)}(\nu^{(2)}) - (1-\theta) p(\beta(1-\theta))\\
&\qquad 
+ \beta \theta (1-\theta) \int_{0}^1 \int_{0}^1 \mathbf{1}_{(-\infty,0)}(x_2-x_1)\, d\nu^{(1)}_X(x_1)\, d\nu^{(2)}_X(x_2)\, ,
\end{align*}
Finally, using (\ref{eq:Ihat}), we may rewrite this as 
\begin{align*}
\mathcal{I}_{\beta}(\nu)\,
&=\, p(\beta)-s(\theta,1-\theta)- \theta p(\beta \theta) - (1-\theta) p(\beta(1-\theta))
+\theta \mathcal{I}_{\beta \theta}(\widehat{\nu}^{(1)}) 
+ (1-\theta) \mathcal{I}_{\beta (1-\theta)}(\widehat{\nu}^{(2)}) \\
&\qquad - \theta S(\nu^{(1)}_X\, |\, \lambda) - \theta S(\nu^{(1)}_Y\, |\, \lambda)
- (1-\theta) S(\nu^{(2)}_X\, |\, \lambda) - (1-\theta) S(\nu^{(2)}_Y\, |\, \lambda)\\
&\qquad 
+ \beta \theta (1-\theta) \int_{0}^1 \int_{0}^1 \mathbf{1}_{(-\infty,0)}(x_2-x_1)\, d\nu^{(1)}_X(x_1)\, d\nu^{(2)}_X(x_2)\, ,
\end{align*}
Using (\ref{eq:yMar}), we may calculate $S(\nu^{(1)}_Y\, |\, \lambda)$ and
$S(\nu^{(2)}_Y\, |\, \lambda)$.
This yields the desired result.
\end{proof}

\begin{corollary}
\label{cor:2square}
For any measures $\mu,\widetilde{\mu} \in \mathcal{M}_{+,1}([0,1])$,
both of which are absolutely continuous with respect to $\lambda$,
we have, for each $\beta \in \R$ and each $\theta \in (0,1)$,
\begin{multline}
\label{ineq:2square0}
- \theta S(\mu\, |\, \lambda)
-(1-\theta) S(\widetilde{\mu}\, |\, \lambda)
+\beta \theta(1-\theta) \int_{0}^{1} \int_{0}^{1}
\mathbf{1}_{(-\infty,0)}(x_2-x_1)\, d{\mu}(x_1)\, d\widetilde{\mu}(x_2)\\
\geq\, 
\theta p(\beta \theta) + (1-\theta) p(\beta(1-\theta)) - p(\beta)\, .
\end{multline}
Moreover, there does exist a pair of measures $\mu,\widetilde{\mu} \in \mathcal{M}_{+,1}([0,1])$ giving equality.
\end{corollary}
\begin{proof}
Recall that for each $\beta \in \R$, we do know that there exists at least one measure in $\mathcal{M}_{+,1}([0,1]^2)$
which minimizes $\mathcal{I}_{\beta}$, using soft analysis, especially lower semi-continuity and weak-compactness. For each $\beta \in \R$, we choose one such measure and call it $\widetilde{\nu}^*_{\beta}$.

Let us define $\kappa,\widetilde{\kappa} \in \mathcal{M}_{+,1}([0,1])$ by
$d\kappa(y) = \theta^{-1} \mathbf{1}_{[0,\theta]}(y)\, dy$
and $d\widetilde{\kappa}(y) = (1-\theta)^{-1} \mathbf{1}_{(\theta,1]}(y)\, dy$.
Then we may define two measures $\xi,\widetilde{\xi} \in \mathcal{M}([0,1]^2)$
by $\xi = \mathfrak{N}(\widetilde{\nu}^*_{\beta\theta},F_{\mu},F_{\kappa})$,
where $F_{\mu}$ and $F_{\kappa}$ are the distribution functions for $\mu$ and $\kappa$, respectively,
and with a similar definition for $\widetilde{\xi}$ based on $\widetilde{\nu}^*_{\beta(1-\theta)}$, $\widetilde{\mu}$ and $\widetilde{\kappa}$.
Then taking $\nu = \theta \xi + (1-\theta) \widetilde{\xi}$, we will have $\nu^{(1)} = \xi$, $\nu^{(2)} = \widetilde{\xi}$, et cetera.
In particular, we will have $\mathcal{I}_{\beta\theta}(\widehat{\nu}^{(1)})=
\mathcal{I}_{\beta(1-\theta)}(\widehat{\nu}^{(2)})=0$
because $\mathcal{I}_{\beta\theta}(\widetilde{\nu}^*_{\beta \theta}) = 0$ and $\mathcal{I}_{\beta(1-\theta)}(\widetilde{\nu}^*_{\beta(1-\theta)})=0$.
So applying Lemma \ref{lem:2square} and using Proposition \ref{prop:Ellis}, we obtain the inequality.

To prove that there does exist a case of equality,
take $\nu \in \mathcal{M}_{+,1}([0,1]^2)$ to be $\widetilde{\nu}^*_{\beta}$, which we know exists
since we do know that a minimizer 
for $\mathcal{I}_{\beta}$ does exist.
Then taking the $\nu_X^{(1)}$ and $\nu_X^{(2)}$ as at the beginning of the section,
we obtain equality.
\end{proof}

\begin{remark}
We note that the proof of the corollary implies that taking $\nu$ to be the optimizer $\widetilde{\nu}^*_{\beta}$,
which we recall is possibly only one of multiple optimizers for $\mathcal{I}_{\beta}$, as far as we have proved, so far.
Then the proof implies that $\widehat{\nu}_X^{(1)}$ and $\widehat{\nu}_X^{(2)}$ are potential choices for $\widetilde{\nu}^*_{\beta \theta}$ 
and $\widetilde{\nu}^*_{(1-\theta)\beta}$. We need not make such choices, explicitly, though.
We point out that this formula is satisfied for the density which is proposed as the unique optimizer.
This can be checked using calculus, for instance.
We leave this exercise to any interested reader (since it is not necessary for the present proof).

We will eventually prove that there is a unique optimizer and that it has the proposed density.
But we will not check that particular calculus fact, explicitly, as such.
The reason is as stated in Remark \ref{rem:ForRef3}.
The inequality is true in general, as has been proved.
One could substitute the known formula for the optimizer to check equality.
(We thank a referee for pointing this out.)
But the direction of the implication in the present proof is that for any optimizer, the inequality must be saturated.
Therefore, we will ultimately conclude from these types of inequalities that the optimizer is unique and has the proposed form.
\end{remark}

\section{The four-square problem}

Suppose that $\theta_1,\theta_2 \in (0,1)$.
In order to simplify notation, let us denote $\Lambda_{ij}(\theta_1,\theta_2)$ just as $\Lambda_{ij}$, for this section.
We assume that we have $(t_{11},t_{12},t_{21},t_{22}) \in \Sigma_4$, defined.
Let us assume, for now, that all four numbers are positive.
We will say later what changes if some $t_{ij}$ equals $0$.

We consider $\nu \in W_{\theta_1,\theta_2}(t_{11},t_{12},t_{21},t_{22})$.
We define four measures, $\nu^{(i,j)}$ for $i,j \in \{1,2\}$ by
\begin{equation}
\label{eq:nuijDecomp}
d\nu^{(i,j)}(x,y)\, =\, t_{ij}^{-1} \mathbf{1}_{\Lambda_{ij}}(x,y)\, d\nu(x,y)\, .
\end{equation}

\begin{lemma}
\label{lem:4square}
Assuming that $\nu \ll \lambda^{\otimes 2}$ and $\nu_X=\nu_Y=\lambda$, using the notation above,
\begin{equation}
\label{eq:calIformula}
\begin{split}
\mathcal{I}_{\beta}(\nu)\,
&=\, p(\beta) + t_{11} \ln(t_{11}) 
+  t_{12} \ln(t_{12}) + t_{21} \ln(t_{21}) + t_{22} \ln(t_{22})\\
&\qquad + \sum_{i,j=1}^{2} \left(- t_{ij} S(\nu^{(i,j)}_X\, |\, \lambda)
-t_{ij} S(\nu^{(i,j)}_Y\, |\, \lambda) + t_{ij} \mathcal{I}_{\beta t_{ij}}(\widehat{\nu}^{(i,j)})
-t_{ij} p(\beta t_{ij})\right)\\
&\qquad + \beta t_{11} t_{12} \int_{0}^{1} \int_{0}^{1} \mathbf{1}_{(-\infty,0)}(x_2-x_1)\, 
d\nu_X^{(1,1)}(x_1)\, d\nu_X^{(1,2)}(x_2)\\
&\qquad + \beta t_{11} t_{21} \int_{0}^{1} \int_{0}^{1} \mathbf{1}_{(-\infty,0)}(y_2-y_1)\,
d\nu^{(1,1)}_Y(y_1)\, d\nu^{(2,1)}_Y(y_2)\\
&\qquad + \beta t_{21} t_{22} \int_{0}^{1} \int_{0}^{1} \mathbf{1}_{(-\infty,0)}(x_2-x_1)\,
d\nu^{(2,1)}_X(x_1)\, d\nu^{(2,2)}_X(x_2)\\
&\qquad + \beta t_{12} t_{22} \int_{0}^{1} \int_{0}^{1} \mathbf{1}_{(-\infty,0)}(x_2-x_1)\,
d\nu^{(1,2)}_Y(y_1)\, d\nu^{(2,2)}_Y(y_2)\\
&\qquad + \beta t_{12} t_{21}\, .
\end{split}
\end{equation}
\end{lemma}
\begin{proof}
One goes through the same steps as in the proof of Lemma \ref{lem:2square}.
\end{proof}

We want to use this to prove Theorem \ref{thm:4square}.
The idea of the completion of the proof is to use Corollary \ref{cor:2square}.
But first we generalize it slightly.

\begin{corollary}
\label{cor:2squareB}
Suppose $\nu_X^{(1)},\nu_X^{(2)} \in \mathcal{M}_{+,1}([0,1])$,
both have support inside an interval $[a,b]$ with $b-a=\theta \in (0,1)$.
Then, for each $\beta \in \R$ and each $t_1,t_2 \in (0,1)$,
\begin{multline}
\label{ineq:2square}
- t_1 S(\nu_X^{(1)}\, |\, \lambda)
- t_2 S(\nu_X^{(2)}\, |\, \lambda)
+\beta t_1 t_2 \int_{0}^{1} \int_{0}^{1}
\mathbf{1}_{(-\infty,0)}(x_2-x_1)\, d\nu_X^{(1)}(x_1)\, d\nu_X^{(2)}(x_2)\\
\geq\, -(t_1+t_2) \ln(\theta)
+t_1 p(\beta t_1) + t_2 p(\beta t_2) - (t_1+t_2) p(\beta(t_1+t_2))\, .
\end{multline}
Moreover, there does exist a pair of measures $\nu_X^{(1)},\nu_X^{(2)} \in \mathcal{M}_{+,1}([0,1])$,
both having support inside $[a,b]$, giving equality.
\end{corollary}
\begin{proof}
This follows from Corollary \ref{cor:2square} by making some scaling transformations.
We may define $\widetilde{\nu}_X^{(1)}$ and $\widetilde{\nu}_X^{(2)}$ by
$$
\widetilde{\nu}_X^{(i)}(A)\, =\, \nu_X^{(i)}(\{a+(b-a)x\, :\, x \in A\})\, .
$$
It is easy to see that $S(\nu_X^{(i)}\, |\, \lambda) = \ln(\theta) + S(\widetilde{\nu}_X^{(i)}\, |\, \lambda)$.
With this, the result follows by straightforward calculations.
\end{proof}

Now we may prove Theorem \ref{thm:4square}, including the case where some $t_{ij}$ may equal $0$.

\begin{proofof}{\bf Proof of Theorem \ref{thm:4square}:}
At first, assume that all $t_{ij}$'s are strictly positive, as we have been assuming in this section, up to this point.
Combining Lemma \ref{lem:4square} with Corollary \ref{cor:2squareB}, and writing $|\Lambda_{ij}|$ for $\lambda^{\otimes 2}(\Lambda_{ij})$, we obtain
\begin{equation}
\label{eq:4squareB}
\begin{split}
\mathcal{I}_{\beta}(\nu)\,
&\geq\, p(\beta) 
+ \sum_{i,j=1}^{2} t_{ij} \ln\left(\frac{t_{ij}}{|\Lambda_{ij}|}\right) \\
&\qquad + \sum_{i,j=1}^{2} t_{ij} \mathcal{I}_{\beta t_{ij}}(\widehat{\nu}^{(i,j)}) +
\sum_{i,j=1}^{2} t_{ij} p(\beta t_{ij}) \\
&\qquad - (t_{11}+t_{12}) p(\beta (t_{11}+t_{12})) - (t_{11}+t_{21}) p(\beta (t_{11}+t_{21}))\\
&\qquad
- (t_{12}+t_{22}) p(\beta (t_{12}+t_{22})) - (t_{21}+t_{22}) p(\beta (t_{21}+t_{22}))\\
&\qquad + \beta t_{12} t_{21}\, .
\end{split}
\end{equation}
Finally, using Proposition \ref{prop:Ellis}, we obtain
\begin{equation}
\label{ineq:4squareC}
\begin{split}
\mathcal{I}_{\beta}(\nu)\,
&\geq\, p(\beta) 
+ \sum_{i,j=1}^{2} t_{ij} \ln\left(\frac{t_{ij}}{|\Lambda_{ij}|}\right) 
+
\sum_{i,j=1}^{2} t_{ij} p(\beta t_{ij}) \\
&\qquad - (t_{11}+t_{12}) p(\beta (t_{11}+t_{12})) - (t_{11}+t_{21}) p(\beta (t_{11}+t_{21}))\\
&\qquad
- (t_{12}+t_{22}) p(\beta (t_{12}+t_{22})) - (t_{21}+t_{22}) p(\beta (t_{21}+t_{22}))\\
&\qquad + \beta t_{12} t_{21}\, .
\end{split}
\end{equation}
We use Proposition \ref{prop:Ellis} to lower bound $\mathcal{I}_{\beta t_{ij}}(\widehat{\nu}^{(i,j)})$
by $0$ everywhere.

The cases of equality follow from the cases of equality in Corollary \ref{cor:2squareB}
as well as the fact that $\tilde{\nu}_{\cdot}^*$ measures do exist to give the minimum in $\mathcal{I}_{\cdot}$.
There is no obstruction to putting these together using $\mathfrak{N}$
somewhat similarly to what we did in the proof of \ref{cor:2square} but with 4 squares instead of 2. 
Corollary \ref{cor:2squareB} give marginals,
and $\tilde{\nu}_{\cdot}^*$ gives values for $\widehat{\nu}$'s.
Then we may use Definition \ref{def:frakN} as a prescription for obtaining $\nu^{(i,j)}$.
See Subsection \ref{subsec:demo} for a demonstration.

Finally, for the case that some $t_{ij}$ equals $0$, all that happens is that the actual value of $\nu^{(i,j)}$ is irrelevant.
In particular note that there is no source of discontinuity of $S$ arising from this.
The Lebesgue measure $|\Lambda_{ij}|$ is fixed and positive because $(\theta_1,\theta_2) \in (0,1)^2$.
All that happens is that the density becomes zero. But $\phi(x) = -x \ln(x)$ is continuous at $0$, since it is defined to be $\phi(0)=0$.
\end{proofof}

We next demonstrate the reconstruction procedure using $\mathfrak{N}$.
This serves two purposes.
First it completes the exercise stated in the proof.
Second it is nice to have a more explicit guide to the construction
later, after we establish uniqueness.
Then we will be able to substitute actual formulas for $\nu_{\beta}$, instead of an existential 
$\widetilde{\nu}^*_{\beta}$.

\subsection{Demonstration of reconstruction}
\label{subsec:demo}

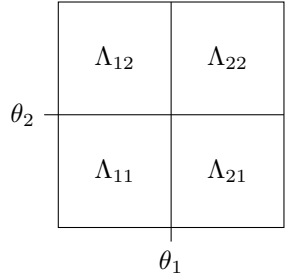
\begin{figure}
    \centering
    \begin{tikzpicture}[xscale=1.5,yscale=1.5]
    \draw (0,0) rectangle (2,2);
    \draw (-0.125,1) -- (2,1);
    \draw (1,-0.125) -- (1,2);
    \draw (0.5,0.5) node[] {$\Lambda_{11}$};
    \draw (1.5,0.5) node[] {$\Lambda_{21}$};
    \draw (0.5,1.5) node[] {$\Lambda_{12}$};
    \draw (1.5,1.5) node[] {$\Lambda_{22}$};
    \draw (1,-0.125) node[below] {$\theta_1$};
    \draw (-0.125,1) node[left] {$\theta_2$};
    \end{tikzpicture}
    \caption{Decomposition of $[0,1]^2$ into $\Lambda_{ij}=\Lambda_{ij}(\theta_1,\theta_2)$ for $i,j\in\{1,2\}$.}
    \label{fig:pic1}
\end{figure}
Let us abbreviate $\Lambda_{ij}(\theta_1,\theta_2)$ by just $\Lambda_{ij}$.
We refer to Figure \ref{fig:pic1} for a reminder of the decomposition defined in (\ref{eq:defLambbda}).

We remind ourselves that we are seeking a $\nu$ which makes the inequality into an equality. We decompose $\nu$
as 
\begin{equation*}
    \nu\, =\, \sum_{i,j=1}^{2} t_{ij} \nu^{(ij)}\, ,
\end{equation*}
as is consistent with (\ref{eq:nuijDecomp}). We note that $\nu^{(ij)}([0,1]^2) = \nu^{(ij)}(\Lambda_{ij})=1$
so that the measures are all normalized to be probability measures.
Also, $\nu^{(ij)}$ has support in $\Lambda_{ij}$.

\subsubsection{General program}

Let us review the basic problem.

We will reconstruct $\nu^{(ij)}$ in steps: first obtaining the $x$ and $y$ marginals, $\widetilde{\nu}^{(ij)}_X$ and $\widetilde{\nu}^{(ij)}_Y$;
then obtaining the standardized measures $\widehat{\nu}^{(ij)}$.
But for reasons of simplicity, we will rescale the arguments of these measures, 
such that $\widetilde{\nu}^{(ij)}_X$ and $\widetilde{\nu}^{(ij)}_Y$
have natural support $[0,1]$ (instead of $L_1(\theta_1)$ and $L_2(\theta_2)$) and $\widehat{\nu}^{(ij)}$ has support $[0,1]^2$ (instead of $\Lambda_{ij}(\theta_1,\theta_2)$).

Then we may integrate these together into $\mathfrak{N}(\widehat{\nu}^{(ij)},\widetilde{F}^{(ij)}_X,\widetilde{F}^{(ij)}_Y)$,
where $\widetilde{F}^{(ij)}_X$ and $\widetilde{F}^{(ij)}_Y$ are the distribution functions for 
$\widetilde{\nu}^{(ij)}_X$ and $\widetilde{\nu}^{(ij)}_Y$.
This is still not $\nu^{(ij)}$, but it is close.
If we denote it by $\widetilde{\nu}^{(ij)}$, then all that is needed is to  rescale the argument so that the support is returmed back to $\Lambda_{ij}$.
In other words
$$
\nu^{(ij)}(A)\, =\, \widetilde{\nu}^{(ij)}(\{(x,y)\, :\, (a_i^{(1)} + (b_i^{(1)}-a_i^{(1)}) x,a_j^{(2)}+(b_j^{(2)}-a_j^{(2)})y) \in A \cap \Lambda_{ij}\})\, ,
$$
where we write $a_i^{(1)}  =\inf(L_i(\theta_1))$ and $b_i^{(1)}=\sup(L_i(\theta_1))$ are the endpoints of $L_i(\theta_1)$,
and similarly
$a_j^{(2)}  =\inf(L_j(\theta_2))$ and $b_j^{(2)}=\sup(L_j(\theta_2))$ are the endpoints of $L_j(\theta_2)$.

Having stated the general program of solution. We now state how to choose optimizers of the various inequalities,
which saturate in the sense that the inequalities are actually equalities.

\subsubsection{The two square problems and the marginals}

We now consider the two square problems, which give the marginals. There are 4 separate two-square problems.
\begin{itemize}
    \item The two-square problem for $\Lambda_{1,1}$ and $\Lambda_{2,1}$ which comprise the bottom horizontal block.
    The solution involves the $y$ marginals for $\nu^{(1,1)}$
    and $\nu^{(2,1)}$.
    \item The two-square problem for $\Lambda_{1,1}$ and $\Lambda_{1,2}$ which comprise the left vertical block.
    The solution involves the $x$ marginals for $\nu^{(1,1)}$
    and $\nu^{(1,2)}$.
    \item The two-square problem for $\Lambda_{1,2}$ and $\Lambda_{2,2}$ which comprise the top horizontal block.
    The solution involves the $y$ marginals for $\nu^{(1,2)}$
    and $\nu^{(2,2)}$.
    \item The two-square problem for $\Lambda_{2,1}$ and $\Lambda_{2,2}$ which comprise the right vertical block.
    The solution involves the $x$ marginals for $\nu^{(2,1)}$
    and $\nu^{(2,2)}$.
\end{itemize}

\underline{\em First two-square problem:}
Starting with the first two-square problem, we can determine some valid choices for $\widetilde{\nu}^{(1,1)}_Y$
and $\widetilde{\nu}^{(2,1)}_Y$.
In obtaining equation (\ref{eq:4squareB}) from (\ref{eq:calIformula}), we used Corollary \ref{cor:2squareB} for each of the 4
distinct two-square problems.
As pertains to the 1st of these two-square problems, we used the inequality to replace 
\begin{equation*}
    -t_{1,1} S(\nu^{(1,1)}_Y\, |\, \lambda) - t_{2,1} S(\nu^{(2,1)}_Y\, |\, \lambda) 
    + \beta t_{11} t_{21} \int_{0}^{1} \int_{0}^{1} \mathbf{1}_{(-\infty,0)}(y_2-y_1)\,
        d\nu^{(1,1)}_Y(y_1)\, d\nu^{(2,1)}_Y(y_2)
\end{equation*}
by its lower bound
\begin{equation*}
    -(t_{11}+t_{21}) \ln(\theta_1) + t_{11} p(\beta t_{11}) + t_{21} p(\beta t_{21}) - (t_{11}+t_{21}) p(\beta (t_{11}+t_{21}))\, .
\end{equation*}
The inequality, in turn, follows from 
by rescaling the supports of  $\nu^{(1,1)}_Y$ and $\nu^{(2,1)}_Y$ to get $\widetilde{\nu}^{(1,1)}_Y$ and $\widetilde{\nu}^{(2,1)}_Y$,
and then applying Corollary \ref{cor:2square} to these.

We replace $\mu$ by $\widetilde{\nu}^{(1,1)}_Y$ and $\widetilde{\mu}$ by $\widetilde{\nu}^{(2,1)}_Y$.
We replace $\beta$ in Corollary \ref{cor:2square} by the new value $(t_{11}+t_{12})\beta$.
Then we replace $\theta$ by $t_{11}/(t_{11}+t_{21})$.
In applying the inequality we eventually multiply every term on both sides by $(t_{11}+t_{21})$ to undo certain divisions.
But that operation is not relevant to us when seeking the measures, themselves, (or rather some of the possibly many measures) which saturate the inequality.

The important thing is that in that corollary we replace $\nu$ by $\widetilde{\nu}^*_{\beta(t_{11}+t_{21})}$ now (instead of $\widetilde{\nu}^*_{\beta}$
as in the original proof, because of our replacement of $\beta$).
We then obtain (some) optimizers, following the procedure outlined in Section \ref{sec:2square}.
Namely, define $\Lambda_1 = [0,t_{11}/(t_{11}+t_{21})]\times [0,1]$ and $\Lambda_2 = (t_{11}/(t_{11}+t_{21}),1]\times [0,1]$
(where we have switched the argument for taking the marginal from $x$ to $y$ just for consistency), and let 
\begin{equation*}
    \nu^{(1)}(\cdot)\, =\, \frac{t_{11}+t_{21}}{t_{11}}\, \cdot \widetilde{\nu}^*_{\beta(t_{11}+t_{21})}(\cdot \cap \Lambda_1)\, ,
\end{equation*}
and
\begin{equation*}
    \nu^{(2)}(\cdot)\, =\, \frac{t_{11}+t_{21}}{t_{21}}\, \cdot \widetilde{\nu}^*_{\beta(t_{11}+t_{21})}(\cdot \cap \Lambda_2)\, .
\end{equation*}
Then we take $\widetilde{\nu}^{(11)}_Y$ to be the $y$ marginal of $\nu^{(1)}$, and we take $\widetilde{\nu}^{(21)}_Y$
to be the $y$ marginal of $\nu^{(2)}$.
This is a choice which makes the inequality in (\ref{ineq:2square}) into an equality for its application in this two-square problem. (This does not itself imply that there are not many other
possible optimizers, as well.)

As an aside, we may note that once we have determined $\widetilde{\nu}^{(11)}_Y$, then we can calculate $\widetilde{\nu}^{(21)}_Y$
directly from it, using 
\begin{equation}
\label{eq:fsq1simp}
    \frac{t_{11}}{t_{11}+t_{21}}\, \cdot \nu^{(1)} + \frac{t_{21}}{t_{11}+t_{21}}\, \cdot \nu^{(2)}\,
    =\, \widetilde{\nu}^*_{\beta(t_{11}+t_{21})}\qquad
    \Rightarrow\quad \frac{t_{11}}{t_{11}+t_{21}}\, \cdot \widetilde{\nu}_Y^{(1,1)} + \frac{t_{21}}{t_{11}+t_{21}}\, \cdot \widetilde{\nu}_Y^{(2,1)}\,
    =\, \lambda\, ,
\end{equation}
because $\nu = \widetilde{\nu}^*_{\beta (t_{11}+t_{21})}$ does have both marginals equal to $\lambda$.
That would allow us to bypass the calculation of the marginal of $\nu^{(2)}$.

So, to summarize, we have that
\begin{equation}
    \label{eq:first2squareSumm}
    \widetilde{\nu}^{(1,1)}_Y(\cdot)\, =\, \frac{t_{11}+t_{21}}{t_{11}}\, \cdot \widetilde{\nu}^*_{\beta(t_{11}+t_{21})}
    (\{(x,y) \in [0,1]^2\, :\, x \in [0,t_{11}/(t_{11}+t_{12})]\, ,\ y \in \cdot])\, .
\end{equation}

\underline{\em Second two-square problem:}
Using the second two-square problem, we obtain a formula which is symmetric to (\ref{eq:first2squareSumm}):
\begin{equation}
    \label{eq:secnd2squareSumm}
    \widetilde{\nu}^{(1,1)}_X(\cdot)\, =\, \frac{t_{11}+t_{12}}{t_{11}}\, \cdot \widetilde{\nu}^*_{\beta(t_{11}+t_{12})}
    (\{(x,y) \in [0,1]^2\, :\, x \in \cdot\, ,\ y \in [0,t_{11}/(t_{11}+t_{21})]])\, .
\end{equation}
Again, this is just a choice for $\widetilde{\nu}^{(1,1)}_X$ (possibly among many other choices, so far as has been determined so far)
which guarantees saturation of the inequality that allowed us to 
replace
\begin{equation*}
    -t_{1,1} S(\nu^{(1,1)}_X\, |\, \lambda) - t_{1,2} S(\nu^{(1,2)}_X\, |\, \lambda) 
    + \beta t_{11} t_{12} \int_{0}^{1} \int_{0}^{1} \mathbf{1}_{(-\infty,0)}(x_2-x_1)\,
        d\nu^{(1,1)}_X(x_1)\, d\nu^{(1,2)}_X(x_2)
\end{equation*}
by its lower bound
\begin{equation*}
    -(t_{11}+t_{12}) \ln(\theta_2) + t_{11} p(\beta t_{11}) + t_{12} p(\beta t_{12}) - (t_{11}+t_{12}) p(\beta (t_{11}+t_{12}))\, .
\end{equation*}
The relationship between $\nu^{(1,1)}_X$ and $\widetilde{\nu}^{(1,1)}_X$ is rescaling the support, as before.
And this is also the relationship between $\nu^{(1,2)}_X$ and $\widetilde{\nu}^{(1,2)}_X$.
Also, $\widetilde{\nu}^{(1,2)}_X$ may be obtained from $\widetilde{\nu}^{(1,1)}_X$
\begin{equation}
\label{eq:fsq2simp}
\frac{t_{11}}{t_{11}+t_{12}}\, \cdot \widetilde{\nu}^{(1,1)}_X + \frac{t_{21}}{t_{11}+t_{21}}\, \cdot \widetilde{\nu}^{(1,2)}_X\,
    =\, \lambda\, ,
\end{equation}
as in (\ref{eq:fsq1simp}).

We can use the other 2 examples of two-square problems to obtain candidates for $\widetilde{\nu}^{(1,2)}_Y$ and $\widetilde{\nu}^{(2,2)}_Y$,
and then $\widetilde{\nu}^{(2,1)}_X$ and $\widetilde{\nu}^{(2,2)}_X$.

\subsubsection{Four square problem}
In going from (\ref{eq:4squareB}) to (\ref{eq:4squareC}), we used
Proposition \ref{prop:Ellis} to lower bound $\mathcal{I}_{\beta t_{ij}}(\widehat{\nu}^{(i,j)})$ by $0$
for all $i,j$.
But from the proposition, we know how to obtain equality in this inequality.
We take an optimizer. So this means we may obtain equality by choosing
\begin{equation}
\label{eq:hatnuForm}
    \widehat{\nu}^{(i,j)}\, =\, \widetilde{\nu}^*_{\beta t_{ij}}
\end{equation}
for all $i,j$.

With this we have the data of all marginals $\widetilde{\nu}^{(i,j)}_X$, $\widetilde{\nu}^{(i,j)}_Y$ and the standardized measures
$\widehat{\nu}^{(i,j)}$.
One just follows the prescription outlined above in the general program.

In Section \ref{sec:Outlook}, we will be more precise, when we give a construction of the optimizer, after first establishing
uniqueness and also calculating several formulas related to the unique choice of $\widetilde{\nu}^*_{\beta}$, namely $\nu^*_{\beta}$
such that $d\nu_{\beta}^*(x,y) = d\rho_{\beta}(x,y)\, dx\, dy$.
For now, we do not proceed any further with this demonstration, instead returning to our main goal of proving uniqueness
and deriving the formula for $\nu^*_{\beta}$ as the unique choice of $\widetilde{\nu}^*_{\beta}$.

\section{Calculus facts}

We have now proved Theorem \ref{thm:4square} and Theorem \ref{thm:Lebesgue}. (We proved them in opposite order.)
The proof of Theorem \ref{thm:main} now occupies us. It follows from calculus exercises.
We will sometimes write $t_{ij}$ for $T_{ij}(\theta_1,\theta_2;t)$, where
$$
T_{11}(\theta_1,\theta_2;t)\, =\, t\, ,\quad
T_{12}(\theta_1,\theta_2;t)\, =\, \theta_1 - t\, ,\quad
T_{21}(\theta_1,\theta_2;t)\, =\, \theta_2 - t\, ,\quad
T_{22}(\theta_1,\theta_2;t)\, =\, 1 - \theta_1 - \theta_2 + t\, .
$$
Let us summarize the main results.

\begin{lemma}
(a) For $t \in \mathcal{I}_{\theta_1,\theta_2}^o = (\max\{0,\theta_1+\theta_2-1\},\min\{\theta_1,\theta_2\})$, we have
\begin{equation}
\label{eq:2ndD}
\frac{\partial^2}{\partial t^2}\, \Phi_{\beta}(\theta_1,\theta_2;t)\, =\, \sum_{i,j=1}^{2} \frac{\beta}{2 \tanh(\beta t_{ij}/2)}\, .
\end{equation}
(b) The critical point equation $\frac{\partial}{\partial t}\Phi_{\beta}(\theta_1,\theta_2;t)=0$ is equivalent to 
\begin{equation}
\label{eq:CritPt}
\frac{(1-e^{-\beta t_{11}})(1-e^{-\beta t_{22}})}{(e^{\beta t_{12}}-1)(e^{\beta t_{21}}-1)}\, =\, 1\, .
\end{equation}
\end{lemma}
\begin{proof}
(a)
Using (\ref{eq:4squareC}) and the definition of $T_{ij}(\theta_1,\theta_2;t)$, for $i,j \in \{1,2\}$, we have 
\begin{equation}
\label{eq:PhiDef}
\begin{split}
\Phi_{\beta}(\theta_1,\theta_2;t)\,
&=\, \Phi_{\beta}(\theta_1,\theta_2;T_{11}(\theta_1,\theta_2;t),T_{12}(\theta_1,\theta_2;t),T_{21}(\theta_1,\theta_2;t),T_{21}(\theta_1,\theta_2;t)) \\
&=\, 
p(\beta) - \theta_1 p(\beta \theta_1) - \theta_2 p(\beta \theta_2) - (1-\theta_1) p(\beta(1-\theta_1))
-(1-\theta_2) p(\beta(1-\theta_2))\\
&\qquad
+ \sum_{i,j=1}^{2} t_{ij} \ln\left(\frac{t_{ij}}{|\Lambda_{ij}|}\right) 
+ \sum_{i,j=1}^{2} t_{ij} p(\beta t_{ij})
 + \beta t_{12} t_{21}\, .
\end{split}
\end{equation}
Moreover, from (\ref{eq:pFormula}), we see that
$$
\frac{\partial}{\partial \theta}\, [\theta p(\beta \theta)]\,
=\, \ln\left(\frac{1-e^{-\beta \tau}}{\beta \tau}\right)\, .
$$
Using this with (\ref{eq:PhiDef}), we see that
\begin{equation*}
\frac{\partial}{\partial t}\, \Phi_{\beta}(\theta_1,\theta_2;t)\, 
=\, 
\sum_{i,j=1}^{2} \left[1+ \ln\left(\frac{t_{ij}}{|\Lambda_{ij}|}\right)\right]
\frac{\partial t_{ij}}{\partial t} 
+
\sum_{i,j=1}^{2} \ln\left(\frac{1-e^{-\beta t_{ij}}}{\beta t_{ij}}\right) \frac{\partial t_{ij}}{\partial t}  + \beta t_{12}\, \frac{\partial t_{21}}{\partial t} + \beta t_{21}\, \frac{\partial t_{12}}{\partial t}\, .
\end{equation*}
Substituting in the partial derivatives $\partial t_{ij}/\partial t = \frac{\partial}{\partial t}T_{ij}(\theta_1,\theta_2;t)$, we obtain
\begin{equation*}
\frac{\partial}{\partial t}\, \Phi_{\beta}(\theta_1,\theta_2;t)\,  =\, 
\sum_{i,j=1}^{2} (-1)^{i+j}\ln\left(\frac{t_{ij}}{|\Lambda_{ij}|}\right)
+
\sum_{i,j=1}^{2} (-1)^{i+j} \ln\left(\frac{1-e^{-\beta t_{ij}}}{\beta t_{ij}}\right)
- \beta (t_{12}+t_{21})\, .
\end{equation*}
This may be rewritten as 
\begin{equation}
\label{eq:FirstDerivCalc}
\frac{\partial}{\partial t}\, \Phi_{\beta}(\theta_1,\theta_2;t)\,  =\, 
\sum_{i,j=1}^{2} (-1)^{i+j}\ln\left(\frac{1-e^{-\beta t_{ij}}}{\beta |\Lambda_{ij}|}\right)
- \beta (t_{12}+t_{21})\, .
\end{equation}
So,
taking the second derivative we obtain
\begin{equation*}
\begin{split}
\frac{\partial^2}{\partial t^2}\, \Phi_{\beta}(\theta_1,\theta_2;t)\, 
&=\, 
\sum_{i,j=1}^{2} (-1)^{i+j}\, \frac{\beta e^{-\beta t_{ij}}}{1-e^{-\beta t_{ij}}} 
\cdot \frac{\partial t_{ij}}{\partial t}
- \beta \left(\frac{\partial t_{12}}{\partial t}+\frac{\partial t_{21}}{\partial t}\right)\\
&=\, 
\sum_{i,j=1}^{2} \frac{\beta e^{-\beta t_{ij}}}{1-e^{-\beta t_{ij}}} 
+ 2\beta\, .
\end{split}
\end{equation*}
Simplifying, this does give equation (\ref{eq:2ndD}).
Note that for $\beta=0$ we interpret this as the $\beta \to 0$ limit which is $\sum_{i,j=1}^{2} 1/t_{ij}$.

(b) Equation (\ref{eq:FirstDerivCalc}) may be rewritten as 
\begin{equation}
\label{eq:FirstDerivCalc2}
\begin{split}
\frac{\partial}{\partial t}\, \Phi_{\beta}(\theta_1,\theta_2;t)\,  
&=\, \ln\left(\frac{(1-e^{-\beta t_{11}})(1-e^{-\beta t_{22}})}{(1-e^{-\beta t_{12}})(1-e^{-\beta t_{21}})}\right)
- \ln\left(e^{\beta t_{12}} e^{\beta t_{21}}\right)\\
&=\, \ln\left(\frac{(1-e^{-\beta t_{11}})(1-e^{-\beta t_{22}})}{(e^{\beta t_{12}}-1)(e^{\beta t_{21}}-1)}\right)\, .
\end{split}
\end{equation}
Therefore, the critical point equation is (\ref{eq:CritPt}).
\end{proof}
Now let us prove Theorem \ref{thm:main}

\begin{proofof}{\bf Proof of Theorem \ref{thm:main}:}
Firstly, we note that $\frac{\partial^2}{\partial t^2}\Phi_{\beta}(\theta_1,\theta_2;t)$ is manifestly positive on $\mathcal{I}_{\theta_1,\theta_2}^o$. Therefore, $\Phi_{\beta}(\theta_1,\theta_2;\cdot) : \mathcal{I}_{\theta_1,\theta_2} \to \R$ is strictly convex.
Next we attempt to solve (\ref{eq:CritPt}).
We recall that $t_{ij} = T_{ij}(\theta_1,\theta_2;t)$. In particular,
$$
t_{11}\, =\, t\, ,\quad
t_{12}\, =\, \theta_1 - t\, ,\quad
t_{21}\, =\, \theta_2 - t\, ,\quad
t_{22}\, =\, 1-\theta_1-\theta_2+t\, .
$$
So,  if we define
$$
u\, =\, 1-e^{-\beta t_{11}}\, =\, 1-e^{-\beta t}\, .
$$
Then we get
\begin{gather*}
1-e^{-\beta t_{22}}\, =\, 1 - e^{\beta(\theta_1+\theta_2-1)} e^{-\beta t}\, =\, 1 - e^{\beta(\theta_1+\theta_2-1)} + e^{\beta(\theta_1+\theta_2-1)} u\, ,\\
e^{\beta t_{12}} - 1\, =\, e^{\beta \theta_1} e^{-\beta t} - 1\, =\, e^{\beta \theta_1} (1-e^{-\beta \theta_1}-u)\, ,\\
e^{\beta t_{21}} - 1\, =\, e^{\beta \theta_2} e^{-\beta t} - 1\, =\, e^{\beta \theta_2} (1-e^{-\beta \theta_2}-u)\, .
\end{gather*}
So (\ref{eq:CritPt}) is equivalent to
\begin{equation}
\frac{u \left(1 - e^{\beta(\theta_1+\theta_2-1)} + e^{\beta(\theta_1+\theta_2-1)} u\right)}
{ e^{\beta (\theta_1+\theta_2)}(1-e^{-\beta \theta_1}-u)(1-e^{-\beta \theta_2}-u)}\, =\, 1\, .
\end{equation}
The equation is equivalent to 
$$
u \left(e^{-\beta(\theta_1+\theta_2)} - e^{-\beta} + e^{-\beta} u\right)\,
=\, (1-e^{-\beta \theta_1}-u)(1-e^{-\beta \theta_2}-u)\, .
$$
Doing one more step of simplication, we obtain the equivalent formulation
$$
(1-e^{-\beta}) u^2 - [(1-e^{-\beta \theta_1}) + (1-e^{-\beta \theta_2}) + e^{-\beta \theta_1} e^{-\beta \theta_2} - e^{-\beta}] u  + (1-e^{-\beta \theta_1})(1-e^{-\beta \theta_2})\, =\, 0\, .
$$
We may simplify this as
$$
(1-e^{-\beta}) u^2 - [(1-e^{-\beta \theta_1})(1-e^{-\beta \theta_2}) + (1 - e^{-\beta})] u  + (1-e^{-\beta \theta_1})(1-e^{-\beta \theta_2})\, =\, 0\, .
$$
Or, splitting the polynomial,
$$
\left[(1-e^{-\beta}) u - (1-e^{-\beta \theta_1})(1-e^{-\beta \theta_2})\right](u-1)\, =\, 0\, .
$$
There are two solutions in the complex plane: $u=1$ which will not correspond to $u = 1-e^{-\beta t}$ for any $t \in \mathcal{I}_{\theta_1,\theta_2}$,
and
$$
u\, =\, \frac{(1-e^{-\beta \theta_1})(1-e^{-\beta \theta_2})}{1-e^{-\beta}}\, .
$$
This leads to the formula (\ref{eq:Rdefin}).

We should check that $t = R_{\beta}(\theta_1,\theta_2)$ is in $\mathcal{I}_{\theta_1,\theta_2}$.
First, we may calculate
\begin{equation}
\label{eq:rMixedDef}
\frac{\partial}{\partial \theta_2}\, R_{\beta}(\theta_1,\theta_2)\, =\, r^{\mathrm{cdf},\mathrm{pdf}}_{\beta}(\theta_1,\theta_2)\, =\,
\frac{e^{-\beta \theta_2} (1-e^{-\beta \theta_1})}{(1-e^{-\beta}) - (1-e^{-\beta \theta_1})(1-e^{-\beta \theta_2})}\, .
\end{equation}
It is easy to see that 
$R_{\beta}(\theta_1,0) = 0$ for all $\theta_1 \in (0,1)$. Therefore, we do have
\begin{equation}
\label{eq:firstIntegral}
R_{\beta}(\theta_1,\theta_2)\, =\, \int_0^{\theta_2} r^{\mathrm{cdf},\mathrm{pdf}}_{\beta}(\theta_1,y)\, dy\, .
\end{equation}
The second derivative calculation is slightly more involved, involving fractions:
\begin{align*}
\frac{\partial}{\partial \theta_1}\, r^{\mathrm{cdf},\mathrm{pdf}}_{\beta}(\theta_1,\theta_2)\, 
&=\,
\frac{\partial}{\partial \theta_1}\, \frac{e^{-\beta \theta_2} (1-e^{-\beta \theta_1})}{(1-e^{-\beta}) - (1-e^{-\beta \theta_1})(1-e^{-\beta \theta_2})}\\
&=\,
e^{-\beta \theta_2}\, \frac{\partial}{\partial \theta_1}\, \frac{ (1-e^{-\beta \theta_1})}{(1-e^{-\beta}) - (1-e^{-\beta \theta_1})(1-e^{-\beta \theta_2})}\\\
&=\,
\beta e^{-\beta \theta_2} e^{-\beta \theta_1}\Big(\frac{1}{(1-e^{-\beta}) - (1-e^{-\beta \theta_1})(1-e^{-\beta \theta_2})}\\[5pt]
&\hspace{4cm}
+ \frac{ (1-e^{-\beta \theta_1})(1-e^{-\beta \theta_2})}{[(1-e^{-\beta}) - (1-e^{-\beta \theta_1})(1-e^{-\beta \theta_2})]^2}\Big)\\
&=\, \frac{\beta (1-e^{-\beta})e^{-\beta \theta_1} e^{-\beta \theta_2}}{[(1-e^{-\beta}) - (1-e^{-\beta \theta_1})(1-e^{-\beta \theta_2})]^2}\, .
\end{align*}
This is manifestly positive for all $\beta \in \R$. (At $\beta=0$ this equals $1$ by taking limits.)
Moreover, from (\ref{eq:rMixedDef})
$r^{\mathrm{cdf},\mathrm{pdf}}_{\beta}(0,\theta_2) = 0$.
So
$$
r^{\mathrm{cdf},\mathrm{pdf}}_{\beta}(\theta_1,\theta_2)\, 
=\, \int_0^{\theta_1} \frac{\beta (1-e^{-\beta})e^{-\beta x} e^{-\beta \theta_2}}{[(1-e^{-\beta}) - (1-e^{-\beta x})(1-e^{-\beta \theta_2})]^2}\, dx\, .
$$
Putting this together with (\ref{eq:firstIntegral}), we see that
$$
R_{\beta}(\theta_1,\theta_2)\, =\, \int_{[0,\theta_1]\times [0,\theta_2]}
\frac{\beta (1-e^{-\beta})e^{-\beta x} e^{-\beta y}}{[(1-e^{-\beta}) - (1-e^{-\beta x})(1-e^{-\beta y})]^2}\, dx\, dy\, .
$$
Direct calculation of $R_{\beta}(\theta_1,1)$ and $R_{\beta}(1,\theta_2)$ shows that it does have the correct marginals.
So, it does follow that $t = R_{\beta}(\theta_1,\theta_2)$ is in $\mathcal{I}_{\theta_1,\theta_2}^o$.
\end{proofof}

\section{Outlook and extensions}
\label{sec:Outlook}

The result that we have presented here is stronger than a weak law of large numbers that was previously proved by one of the authors \cite{Starr}.
Moreover, the present argument is simpler, since the old argument used uniqueness theory for a certain type of partial differential equation.
The old proof was not direct.

However, the old result, weak as it is, was useful in a subsequent work by Mueller and one of the authors \cite{MuellerStarr}.
That was a weak law for the length of the longest increasing subsequence in a Mallows distributed random permutation, when $q=q_n$
scales such that $1-q_n \sim \beta/n$ as $n \to \infty$, for some $\beta \in \R$.
Bhatnagar and Peled considered the more general case that $q_n$ may scale with $n$ in a more singular way.
It is therefore interesting to look for a more exact type of result than what has been presented in this article.

The following result is true, and we will prove it in some detail in Appendix \ref{app:fourSQUAREProof}.
\begin{lemma}
\label{lem:fourSQUARE}
Let us define $\{n\}!\, :=\, [n]_q!/n!$. Let us define
$$
\Sigma_4(n)\, =\, \{(n_{11},n_{12},n_{21},n_{22}) \in \{0,1,\dots\}^4\, :\, n_{11}+n_{12}+n_{21}+n_{22}=n\}\, .
$$
Then for each such 4-tuple, defining $\mathcal{P}_{n,\beta}(n_{11},n_{12},n_{21},n_{22};\theta_1,\theta_2)$ to be
$$
\mu_{n,\beta}\left(\left\{\big((x_1,y_1),\dots,(x_n,y_n)\big) \in ([0,1]^2)^n\, :\, 
\frac{1}{n} \sum_{k=1}^n \delta_{(x_k,y_k)} \in W_{\theta_1,\theta_2}\left(\frac{n_{11}}{n},\frac{n_{12}}{n},\frac{n_{21}}{n},\frac{n_{22}}{n}\right)\right\}\right)\, ,
$$
we have the dependence on $\beta$, versus the usual multinomial formula for $\beta=0$:
\begin{align*}
\mathcal{P}_{n,\beta}(n_{11},n_{12},n_{21},n_{22};\theta_1,\theta_2)
&=\, \mathcal{P}_{n,0}(n_{11},n_{12},n_{21},n_{22};\theta_1,\theta_2)\\
&\hspace{-1cm} \times
q^{n_{12} n_{21}}\,
\frac{\{n_{11}+n_{12}\}!\{n_{11}+n_{21}\}!\{n_{12}+n_{22}\}!\{n_{21}+n_{22}\}!}
{\{n_{11}\}!\{n_{12}\}!\{n_{21}\}!\{n_{22}\}!\{n_{11}+n_{12}+n_{21}+n_{22}\}!}\Bigg|_{q=\exp(-\beta/(n-1))}\, .
\end{align*}
\end{lemma}
This may be proved using similar symmetries to those already on display here, except at the discrete level.

This leads to a quantitative version of the results on display here, sufficient to establish a local central limit theorem.

But moreover, it may be possible that this applies for $q=q_n$ scaling with $n$ in a more singular way than $1-q_n \sim c/n$.
More precisely, Moak has obtained a full asymptotic expansion for the $q$-factorial numbers such as $\{n\}_q$ in \cite{Moak}.
When $q=\exp(-\beta/(n-1))$ the leading order part does lead to the large deviation function formulas we have derived here.
But there are also correction terms, and more generally the lower order terms may be relevant when $q$ is not scaling as $1-\beta n^{-1}$ for some finite $\beta \in \R$.
That might be useful for trying to further analyze models considered by Bhatnagar and Peled in \cite{BhatnagarPeled}.
In a result related to \cite{MuellerStarr}, we are considering a 9-square problem, where the middle square is small, having linear size on the order
of $1/n^{1/4}$.
That analysis is currently underway, and we hope to report on it soon, in another article.

\subsection{Connection to physics}
The connection to physics is also of possible interest.
There is a direct relation between the Mallows measure and the ground states of the $\mathrm{SU}_q(2)$-invariant XXZ model
with Ising-type anisotropy, and kink boundary conditions.
This has been described in \cite{Starr} in detail, so let us summarize quickly, here. 

Actually,
the quickest summary is in terms of the invariant measure for the asymmetric exclusion process on the linear chain $\{1,\dots,n\}$
with no boundary conditions: i.e., no birth or death of particles at the two ends.
Then the particle number is conserved and there are a number of ergodic components.
There are $n+1$ invariant measures, corresponding to possible particle numbers.
One may write them down in terms of the Mallows measure.
This is connected to the XXZ model because the ASEP and the XXZ model are equivalent on the linear chain $\{1,\dots,n\}$,
with the boundary conditions prescribed by the words, above.
Let us quickly describe the equivalence without writing down formulas for the Hamiltonian of the XXZ model
or the generator of the ASEP.

Let $\mathcal{X}_n$ denote the set of all occupation functions $\eta : \{1,\dots,n\} \to \{0,1\}$.
For $k \in \{0,1,\dots,n\}$,
let $\mathcal{X}_{n,k}$
denote the set of all occupation functions $\eta$ such that $\sum_{i=1}^{n} \eta(i)=k$.
Let $(P^t)_{t\geq 0}$ denote the transition kernel for the continuous time ASEP after a time $t$,
so that if $\mu:\mathcal{X}_n \to \R$ defines a probability measure for an initial configuration $\eta_0$,
and $f : \mathcal{X}_n \to \R$ is a function, such that $\E^{\mu}[f(\eta_0)] = \sum_{\eta \in \mathcal{X}^n} \mu(\eta) f(\eta)$,
then we have $\E^{\mu}[f(\eta_t)] = \sum_{\eta \in \mathcal{X}^n} \sum_{\eta' \in \mathcal{X}^n} \mu(\eta) P^t(\eta,\eta') f(\eta')$.
Then $X = \frac{d}{dt} P^t\big|_{t=0}$ is the generator of the ASEP, which is a matrix on the vector space $\Omega_n$ whose basis is 
the set of occupation functions $\eta \in \mathcal{X}_n$, which has dimension $2^n$.

The XXZ model is a quantum spin system, for spins (of particles which themselves do not move) with spin $1/2$. 
So, when there are $n$ spins, it is given by a Hermitian matrix acting on a finite-dimensional Hilbert space of dimension $2^n$,
typically thought of as $\Hil_n = \C^2 \otimes \cdots \otimes \C^2$, with $\C^2$ having basis $\ket{\uparrow}$, $\ket{\downarrow}$,
corresponding to physicists' notation for two possible states of a spin-$1/2$ particle (such as an electron's spin).
But, instead if we think of the basis of the $i$th tensor factor as $\ket{\eta(i)}$ for the two possibilities $\eta(i)\in\{0,1\}$,
then there is actually an isomorphism between the two vector spaces $\Omega_n$ and $\Hil_n$.

Under the natural isomorphism $X_n$ is not (necessarily) mapped to the Hermitian matrix which is the Hamiltonian of the system.
But it is equivalent to it, in terms of the spectrum.
Similarly, $P^t$ is not necessarily a symmetric matrix.
But the Markov chain is reversible.
So there is a similarity transformation of $P^t$, as a matrix, which transforms it into a symmetric matrix.
The similarity transformation involves the invariant measures, themselves, on each of the ergodic components.

Now, for the invariant measures, on $\mathcal{X}_{n,k}$, let us define a probability mass function $P'_{n,k,q}$ which is given in
terms of the Mallows model as
\begin{equation*}
    P'_{n,k,q}(\eta)\, =\, \sum_{\pi \in S_n} P_{n,q}(\pi) \mathbf{1}(\{\eta = H^{(k)}(\pi)\})\, ,\quad \text{ where }\
    H^{(k)}_{\pi}(i)\, =\, \mathbf{1}_{\pi(1),\dots,\pi(k)}(j)\, ,\ \text{ for $i \in \{1,\dots,n\}$,}
\end{equation*}
where $P_{n,q}$ is the Mallows measure defined in (\ref{eq:Pdef}).
Then $P'_{n,k,q}$ is the invariant measure for the ASEP with $k$ particles and $n-k$ holes on $\{1,\dots,n\}$
Therefore, by a simple transformation, it is directly related to the ground state of the XXZ Hamiltonian
with $k$ spins $\uparrow$ and $n-k$ spins $\downarrow$.
So questions about the ground state of the XXZ Hamiltonian may be answered in terms of the Mallows measure.

One question that can be asked is this: suppose that there are $k$ spins $\uparrow$, but for some $j \in \{1,\dots,n\}$
we ask that there are exactly $r$ spins $\uparrow$ inside the interval $\{1,\dots,j\}$ (and therefore
$k-r$ spins $\uparrow$ in the complementary interval $\{j+1,\dots,n\}$).
What is the probability of this ``event,''
meaning the expected value of the observable which is diagonal in the spin basis corresponding to the indicator function of this set?
This is precisely the question that is answered by the four-square problem.

Such questions in quantum spin systems are not totally irrelevant.
For example, for the Heisenberg antiferromagnet, the ground state is much more complicated than for the ferromagent.
(We have described the XXZ ferromagnet, above.)
For the antiferromagnet it became an interesting question for physicists to explore the probability that in a given interval there
are 0 spins $\uparrow$.
Vladimir Korepin called this the emptiness formation probability and initiated its study.
Through difficult work, physicists obtained explicit formulas \cite{KorepinEtAl}.
Even rough bounds are somewhat difficult if one demands mathematical rigor \cite{CrawfordNgStarr}.

In comparison, the explicit formulas we obtained are simple, primarily because the Mallows measure is not
related to the antiferromagnet but to the ferromagnet.

Finally, let us state an important fact.
The function $H^{(0)}_{\pi},H^{(1)}_{\pi},\dots,H^{(n)}_{\pi}$ are known as the height functions for the measure $\pi$,
as defined by Wilson in \cite{Wilson}.
Knowing them completely recovers $\pi$.
There is also a Markov chain on $S_n$ such that, using these height functions, the ASEP is a projection.
This Markov chain was defined by Diaconis and Ram \cite{DiaconisRam}, and is a generalization of the stirring process (which
is what it reduces to when $q=1$, see for example \cite{Liggett}).
It is challenging to relate properties of the ASEP, as a stochastic process, back to the Markov chain of Diaconis and Ram.
For example, this is what was done by Benjamini, Berger, Hoffman and Mossel \cite{BBHM}.

One can ask for the simpler problem: to deduce formulas for the Mallows measure from formulas for the invariant measures of the ASEP, which are well-known (see for example \cite{Liggett}).
In particular, one can ask precisely the question that this paper has answered: for a scaling limit of the Mallows measure
using the well-known scaling limit of the invariant measures of the ASEP.
In a private communication one of us (S.S.)\ learned that Wilson had done this, just around the same time
that the genesis of the idea for \cite{Starr} had occurred to us.
It is possible that the present article is just a formulation of Wilson's approach!
(We never inquired about the specific details.)

In any case, Wilson's height functions are certainly the key idea in the background for the present article.
Moreover, they were already used by BBHM to do something much more non-trivial.

\subsection{Explication of the optimizer}

At this point, we may revisit the formulas in Subsection \ref{subsec:demo}, and input the actual formulas.
In other words we can now be more explicit about the unique solution $\nu$ to the optimization problem
\begin{equation*}
    \widetilde{\Phi}_{\beta}(\theta_1,\theta_2;t_{11},t_{12},t_{21},t_{22})\, =\, \min_{\nu \in W_{\theta_1,\theta_2}(t_{11},t_{12},t_{21},t_{22})}
    \mathcal{I}_{\beta}(\nu)\, .
\end{equation*}
We consider $\nu^{(1,1)}$.
From (\ref{eq:first2squareSumm}) we can write
\begin{equation}
    \widetilde{\nu}^{(1,1)}_Y(dy)\, =\, \frac{t_{11}+t_{21}}{t_{11}}\, \cdot \nu^*_{\beta(t_{11}+t_{21})}
    (\{(x,y_1) \in [0,1]^2\, :\, x \in [0,t_{11}/(t_{11}+t_{21})]\, ,\ y_1 \in [y,y+dy)\})\, .
\end{equation}
But this is the density from (\ref{eq:rMixedDef}). So we have
\begin{equation}
\begin{split}
    d\widetilde{\nu}^{(1,1)}_Y(y)\, &=\, 
    \frac{t_{11}+t_{21}}{t_{11}}\, \cdot r^{\mathrm{cdf},\mathrm{pdf}}_{\beta (t_{11}+t_{21})}\left(\frac{t_{11}}{t_{11}+t_{21}}\, ,\ y\right)\, dy\\
    &=\,
    \frac{t_{11}+t_{21}}{t_{11}}\, \cdot \frac{e^{-\beta (t_{11}+t_{21})y} (1-e^{-\beta t_{11}})}{(1-e^{-\beta(t_{11}+t_{21})}) - (1-e^{-\beta t_{11}})(1-e^{-\beta (t_{11}+t_{21})y})}\, dy\, .
\end{split}
\end{equation}
What we really want, however, is the density function $\widetilde{F}^{(1,1)}_Y$.
But by (\ref{eq:firstIntegral}) this is given by $R$:
\begin{equation}
\label{eq:tildeFform1}
\begin{split}
    \widetilde{F}^{(1,1)}_Y(y)\, 
    &=\, 
    \int_0^{y}     \frac{t_{11}+t_{21}}{t_{11}}\, \cdot r^{\mathrm{cdf},\mathrm{pdf}}_{\beta (t_{11}+t_{21})}\left(\frac{t_{11}}{t_{11}+t_{21}}\, ,\ y_1\right)\, dy_1\\
    &=\, 
    \frac{t_{11}+t_{21}}{t_{11}}\, \cdot R_{\beta (t_{11}+t_{21})}\left(\frac{t_{11}}{t_{11}+t_{21}}\, ,\ y\right)\\ 
    &=\,
    -\frac{1}{\beta t_{11}}\, \ln\left(1-\frac{(1-e^{-\beta t_{11}})(1-e^{-\beta (t_{11}+t_{21})y})}{1-e^{-\beta(t_{11}+t_{21})}}\right)
\end{split}
\end{equation}
By symmetry, we can also calculate $\widetilde{F}^{(1,1)}_X(x)$ similarly
\begin{equation}
\label{eq:tildeFform2}
    \widetilde{F}^{(1,1)}_X(x)\, =\,
    -\frac{1}{\beta t_{11}}\, \ln\left(1-\frac{(1-e^{-\beta t_{11}})(1-e^{-\beta (t_{11}+t_{12})x})}{1-e^{-\beta(t_{11}+t_{12})}}\right)
\end{equation}
Now $\widehat{\nu}^{(1,1)}$ is just exactly $\nu^*_{\beta t_{11}}$: in general $\widehat{\nu}^{(i,j)}$ is just exactly $\nu^*_{\beta t_{ij}}$
from (\ref{eq:hatnuForm}).

Now we are supposed to integrate these together using $\mathfrak{N}$.
Let us define $\widetilde{\nu}^{(11)}$ to be $\mathfrak{N}(\widehat{\nu}^{(1,1)},\widetilde{F}^{(1,1)}_X,\widetilde{F}^{(1,1)}_Y)$.
This will not quite be $\nu^{(11)}$. But it is the rescaling of the support of $\nu^{(11)}$ to change the support from $\Lambda_{11}(\theta_1,\theta_2)$
to $[0,1]^2$.
So $\nu^{(1,1)}([0,\theta_1 x] \times [0,\theta_2 y])$ will equal $\widetilde{\nu}^{(1,1)}([0,x]\times [0,y])$.
Then from Definition \ref{def:frakN}, this means
\begin{equation}
\begin{split}
    \nu^{(1,1)}([0,\theta_1 x]\times [0,\theta_2 y])\,
    &=\, R_{\beta t_{11}}(\widetilde{F}^{(1,1)}_X(x),\widetilde{F}^{(1,1)}_Y(y))\\
    &=\, -\frac{1}{\beta t_{11}}\, 
    \ln\left(1-
    \frac{(1-e^{-\beta t_{11}\widetilde{F}^{(1,1)}_X(x)})(1-e^{-\beta t_{11} \widetilde{F}^{(1,1)}_Y(y)})}
    {1-e^{-\beta t_{11}}}\right)\, .
\end{split}
\end{equation}
Because of the formulas (\ref{eq:tildeFform1}) and (\ref{eq:tildeFform2}), this is
\begin{equation}
    \nu^{(1,1)}([0,\theta_1 x]\times [0,\theta_2 y])\,
    =\, -\frac{1}{\beta t_{11}}\, 
    \ln\left(1-
    \frac{(1-e^{-\beta t_{11}})(1-e^{-\beta (t_{11}+t_{12})x})(1-e^{-\beta (t_{11}+t_{21})y})}
    {(1-e^{-\beta (t_{11}+t_{12})})(1-e^{-\beta (t_{11}+t_{21})})}\right)\, .
\end{equation}
This formula can also give the formula for $\nu^{(2,2)}$ by symmetry $x\mapsto 1-x$ and $y\mapsto 1-y$.
So
\begin{multline}
    \nu^{(2,2)}([1-(1-\theta_1)x,1]\times [1-(1-\theta_2)y,1])\\
    =\, -\frac{1}{\beta t_{22}}\, 
    \ln\left(1-
    \frac{(1-e^{-\beta t_{22}})(1-e^{-\beta (t_{21}+t_{22})x})(1-e^{-\beta (t_{12}+t_{22})y})}
    {(1-e^{-\beta (t_{21}+t_{22})})(1-e^{-\beta (t_{12}+t_{22})})}\right)\, .
\end{multline}
In principle, the formulas for $\nu^{(1,2)}$ and $\nu^{(2,1)}$ can also be calculated
from these formulas using the symmetry $\beta \mapsto -\beta$, $x\mapsto 1-x$.
and the symmetry $\beta \mapsto -\beta$, $y \mapsto 1-y$.
We will leave this to the interested reader.

\begin{remark}
We are grateful to an anonymous referee for providing the impetus to write up this guide.
\end{remark}

\appendix

\section{Proof of Lemma \ref{lem:fourSQUARE}}
\label{app:fourSQUAREProof}

Let us begin with a lemma.
\begin{lemma}
\label{lem:Shuffle}
For each $n \in \N$,
the polynomial $P_{n}(q) = \sum_{\pi \in S_n} q^{\operatorname{inv}(\pi)}$ has the formula
$$
P_n(q)\, =\, \prod_{k=1}^{n} \left(\frac{1-q^k}{1-q}\right)\, \stackrel{\mathrm{def}}{=:}\, [n]_q!\, ,
$$
and for any $n \in \N$ and $k \in \{1,\dots,n\}$
$$
\sum_{\pi \in \operatorname{Sh}_{n,k}} q^{\operatorname{inv}(\pi)}\, 
=\, \frac{[n]_q!}{[k]_q! [n-k]_q!}\,
\stackrel{\mathrm{def}}{=:}\, \genfrac[]{0pt}{0}{n}{k}_q\, ,
$$
where $\operatorname{Sh}_{n,k}$ is the set of all permutations $\pi \in S_n$
such that $\pi_1<\dots<\pi_k$ and $\pi_{k+1}<\dots<\pi_n$.
\end{lemma}
\begin{proof}
The first formula is well-known. One may consult \cite{DiaconisRam}, for instance, and references, therein.

The second formula is also well-known.
But let us review the proof, since we will use the same ideas, subsequently.
Given $\pi \in S_n$, and given $k \in \{1,\dots,n\}$, denote by
$\pi^{(k)}$ the permutation such that
$$
\{\pi^{(k)}_1,\dots,\pi^{(k)}_k\}\, =\, \{\pi_1,\dots,\pi_k\}\ \text{ and }\
\{\pi^{(k)}_{k+1},\dots,\pi^{(k)}_{n}\}\, =\, \{ \pi_{k+1},\dots,\pi_n\}\, ,
$$
and such that $\pi^{(k)}$ is an element of $\operatorname{Sh}_{n,k}$.
Also, let $\pi^{(k,-)}$ denote the permutation in $S_n$ such that $\pi^{(k,-)}_{j}=j$
for $j>k$ and $\pi_j = \pi^{(k)}_{\pi^{(k,-)}_j}$ for $j \in \{1,\dots,k\}$.
Similarly, define $\pi^{(k,+)}$ such that
$\pi^{(k,+)}_{j}=j$
for $j\leq k$ and $\pi_j = \pi^{(k)}_{\pi^{(k,-)}_j}$ for $j \in \{k+1,\dots,n\}$.
Then it is an easy fact to see that
$$
\operatorname{inv}(\pi)\, =\, \operatorname{inv}(\pi^{(k)}) + 
 \operatorname{inv}(\pi^{(k,-)})
+  \operatorname{inv}(\pi^{(k,+)})\, .
$$
Moreover, defining $S_{n,k}^-$ to be the set of all permutations $\pi \in S_n$
such that $\pi_j=j$ for all $j>k$ (which is isomorphic to $S_k$)
and $S_{n,k}^+$ to be the set of all permutations $\pi \in S_n$
such that $\pi_j=j$ for all $j\leq k$ (which is isomorphic to $S_{n-k}$),
we have $S_n \cong \operatorname{Sh}_{n,k} \times S_{n,k}^- \times S_{n,k}^+$,
where the bijection is the mapping $\pi \mapsto (\pi^{(k)},\pi^{(k,-)},\pi^{(k,+)})$.
Using this, one may easily prove the second formula.
\end{proof}

With this, we may give the desired proof.

\begin{proofof}{\bf Proof of Lemma \ref{lem:fourSQUARE}:}
We are going to start with a particular construction of a set of points $((x_1,y_1),\dots,(x_n,y_n))$
such that $n^{-1} \sum_{k=1}^{n} \delta_{(x_k,y_k)}$ is in
$W_{\theta_1,\theta_2}((n_{ij}/n)_{i,j=1}^{2})$.

First of all we note that the particular ordering of $((x_1,y_1),\dots,(x_n,y_n))$
does not affect the density $\mu_{n,\beta}((x_1,y_1),\dots,(x_n,y_n))$.
Therefore, we will only describe the construction modulo a permutation in $S_n$.
We return to this point at the end.

Let us denote
$$
n^{(1:1)}_1\, =\, n^{(2:1)}_1\, =\, n_{11}\, ,\ 
n^{(1:1)}_2\, =\, n^{(2:2)}_1\, =\, n_{12}\, ,\ 
n^{(1:2)}_1\, =\, n^{(2:1)}_2\, =\, n_{21}\, ,\
n^{(1:2)}_2\, =\, n^{(2:2)}_2\, =\, n_{22}\, .
$$
Then for each $i,j \in \{1,2\}$
let us define $n^{(i:j)} = n^{(i:j)}_1 + n^{(i:j)}_2$.

Let us choose points
$$
0<\tilde{t}^{(1:1)}_1<\dots<\tilde{t}^{(1:1)}_{n^{(1)}_1}<\theta_1\, ,\qquad
\theta_1<\tilde{t}^{(1:2)}_1<\dots<\tilde{t}^{(1:2)}_{n^{(1)}_2}<1\, ,
$$
and
$$
0<\tilde{t}^{(2:1)}_1<\dots<\tilde{t}^{(2:1)}_{n^{(2)}_1}<\theta_2\, ,\qquad
\theta_2<\tilde{t}^{(2:2)}_1<\dots<\tilde{t}^{(2:2)}_{n^{(2)}_2}<1\, .
$$
Next, for each $i,j \in \{1,2\}$, let us choose a shuffle permutations
$\pi^{(i:j)} \in \operatorname{Sh}_{n^{(i:j)},n^{(i:j)}_1}$.
Then we define 
$$
\tilde{\tilde{t}}^{(i:j:1)}_k\, =\, \tilde{t}^{(i:j)}_{\pi^{(i:j)}_k}\
\text{ for $k \in \{1,\dots,n^{(i:j)}_1\}$, and }\
\tilde{\tilde{t}}^{(i:j:2)}_k\, =\, \tilde{t}^{(i:j)}_{\pi^{(i:j)}_{n^{(i:j)}_1+k}}\
\text{ for $k \in \{1,\dots,n^{(i:j)}_2\}$.}\
$$

Finally, let us choose permutations $\pi^{(i,j)} \in S_{n_{ij}}$ for each $i,j \in \{1,2\}$.
Then we define $Z_{ij} \subset \Lambda_{i,j}(\theta_1,\theta_2)$ with $|Z_{ij}|=n_{ij}$
for each $i,j \in \{1,2\}$, as 
\begin{align*}
Z_{11}\, &=\, \{(\tilde{\tilde{t}}^{(1:1:1)}_k,{\tilde{\tilde{t}}}^{(2:1:1)}_{\pi^{(1,1)}_k})\, :\, 
k = 1,\dots,n_{11}\}\, ,\\
Z_{12}\, &=\, \{(\tilde{\tilde{t}}^{(1:1:2)}_k,{\tilde{\tilde{t}}}^{(2:2:1)}_{\pi^{(1,2)}_k})\, :\, 
k = 1,\dots,n_{12}\}\, ,\\
Z_{21}\, &=\, \{(\tilde{\tilde{t}}^{(1:2:1)}_k,{\tilde{\tilde{t}}}^{(2:1:2)}_{\pi^{(2,1)}_k})\, :\, 
k = 1,\dots,n_{21}\}\, ,\\
Z_{22}\, &=\, \{(\tilde{\tilde{t}}^{(1:2:2)}_k,{\tilde{\tilde{t}}}^{(2:2:2)}_{\pi^{(2,2)}_k})\, :\, 
k = 1,\dots,n_{22}\}\, .
\end{align*}
We let $(x_1,y_1),\dots,(x_n,y_n)$ be an arbitrary choice of points (meaning an arbitrary ordering)
comprising $\bigcup_{i,j=1}^{n} Z_{ij}$.

For each $i,j,I,J \in \{1,2\}$, let us define
$$
\mathcal{N}_{(i,j),(I,J)}((x_1,y_1),\dots,(x_n,y_n))\,
=\, \#\{\{k,\ell\}\, :\, (x_k,y_k) \in \Lambda_{i,j}\, ,\ (x_{\ell},y_{\ell}) \in \Lambda_{I,J}\,
(x_k-x_{\ell})(y_k-y_{\ell})<0\}\, .
$$
Then it is easy to see that
$$
H_n\,
=\, \sum_{i,j=1}^{2} \mathcal{N}_{(i,j),(i,j)} + \mathcal{N}_{(1,1),(1,2)} + \mathcal{N}_{(1,1),(2,1)}
+\mathcal{N}_{(1,2),(2,2)} + \mathcal{N}_{(2,1),(2,2)} + \mathcal{N}_{(1,2),(2,2)}\, .
$$
We do not need to include $\mathcal{N}_{(1,1),(2,2)}$ because if $(x_k,y_k) \in \Lambda_{11}$
and $(x_{\ell},y_{\ell}) \in \Lambda_{22}$ then we must have $(x_k-x_{\ell})(y_k-y_{\ell})>0$.
For similar reasons, we have
$\mathcal{N}_{(1,2),(2,1)}((x_1,y_1),\dots,(x_n,y_n)) = n_{12} n_{21}$.
Finally, we claim that careful inspection shows that
\begin{gather*}
\mathcal{N}_{(i,j),(i,j)}\, =\, \operatorname{inv}(\pi^{(i,j)})\, ,\
\text{ for each $i,j \in \{1,2\}$}\\
\mathcal{N}_{(1,1),(1,2)}\, =\, \operatorname{inv}(\pi^{(1:1)})\, ,\quad
\mathcal{N}_{(1,1),(2,1)}\, =\, \operatorname{inv}(\pi^{(2:1)})\, ,\\
\mathcal{N}_{(1,2),(2,2)}\, =\, \operatorname{inv}(\pi^{(2:2)})\, ,\quad
\mathcal{N}_{(2,1),(2,2)}\, =\, \operatorname{inv}(\pi^{(1:2)})\, .\\
\end{gather*}
Then, using the fact that $Z_n(\beta) = [n]_q!/n!$ for $q=\exp(-\beta/(n-1))$,
we have
$$
\mu_{n,\beta}((x_1,y_1),\dots,(x_n,y_n))\, =\, \frac{n!}{[n]_q!}\,
q^{n_{12}n_{21}}\, \left[\prod_{i,j=1}^{2} q^{\operatorname{inv}(\pi^{(ij)})}\right]
\left[\prod_{i,j=1}^{2} q^{\operatorname{inv}(\pi^{(i:j)})}\right]\, .
$$
Then summing over each of the independent permutations $\pi^{(i,j)} \in S_{n_{ij}}$
for $i,j \in \{1,2\}$ and $\pi^{(i:j)} \in \operatorname{Sh}_{n^{(i:j)},n^{(i:j)}_1}$
for $i,j \in \{1,2\}$, and using Lemma \ref{lem:Shuffle}, we may deduce the result.
\end{proofof}

\section{Another derivation of the main results}

We announce here another approach to derive our main results, without providing all the details.
Firstly, there does exist a $q$-Stirling's formula just like the usual Stirling's formula \cite{Moak},
wherein one may see that the exponential term is as in (\ref{eq:pFormula}).
In fact, the $q$-Stirling's formula may be considered to be superior because the dilogarithm
arises in the leading-order exponential term rather than just in the constant term.
(In the usual
Stirling's formula the constant $\sqrt{2\pi}$ must be derived by one of a number of methods.
One method is to realize that $\ln(n!) - (n+\frac{1}{2})\ln(n)-n+\frac{1}{2} \ln(2e)=R_n$
where $R_n \to \int_0^{1/2} \ln(\pi x/\sin(\pi x))\, dx$.
Modulo elementary functions, this is equal to the dilogarithm function at $-1$, which may be evaluated
as easily as it is to calculate $\zeta(2)$.
On the other hand, since (\ref{eq:pFormula}) is also another expression for the dilogarithm of 
an exponential, we see that the dilogarithm arises in the leading-order exponential part of the $q$-Stirling
formula.)

An explicit formulation is as follows: for each $\beta \in \R$
\begin{equation}
\label{eq:qSTIRLING}
\ln\left(\frac{[n]_q!}{n!}\right)\bigg|_{q=\exp(-\beta/(n-1))}\,
=\, n \int_0^1 \ln\left(\frac{1-e^{-\beta x}}{\beta x}\right)\, dx
+ \frac{\beta}{2}
+ \frac{1}{2}\, \ln\left(\frac{1-e^{-\beta}}{\beta}\right)
+ R_n(\beta)\, ,
\end{equation}
where $R_n(\beta) \to 0$ as $n \to \infty$.
This follows from Moak's formula.
But also, a simple proof is given (by specializing to $q$'s of the form $q_n=1-n^{-1}(\beta  + o(1))$
where $o(1)\to 0$ as $n \to \infty$) in \cite{Walters}: see Theorem 5.1.1.

Using this formula, keeping only the leading-order
exponential term, and using Lemma \ref{lem:fourSQUARE}, we may prove that
\begin{multline*}
\lim_{n \to \infty} \mu_{n,\beta}\left(\left\{((x_1,y_1),\dots,(x_n,y_n))\, :\, \frac{1}{n}\, \sum_{k=1}^{n}
\delta_{(x_k,y_k)} \in W_{\theta_1,\theta_2}((\nu_{ij})_{i,j=1}^{2})\right\}\right)\\
=\, \widetilde{\Phi}_{\beta}(\theta_1,\theta_2;\nu_{11},\nu_{12},\nu_{21},\nu_{22})\, ,
\end{multline*}
the formula from Theorem \ref{thm:4square}.
But this is a harder approach.
It is proved in complete detail in \cite{Walters} in Lemma 5.6.2.
The advantage of this second approach is that by keeping track of the lower order terms in 
(\ref{eq:qSTIRLING})
one may prove local limit theorems.
We intend to do this in a later paper for the 9-square problem,
where the middle square has both sidelengths on the order of $n^{-1/4}$ as a subset of $[0,1]^2$,
so that there are order $n^{1/2}$ points in the middle square.
That is because this is what is needed to obtain the simplest bounds on the fluctuations for 
the length of the longest increasing subsequence in a Mallows random permutation when
$q_n = 1 - n^{-1}(\beta + o(1))$.
The weak limit of this problem was previously considered in \cite{MuellerStarr}.
But we will obtain essentially $O(n^{(1/4)+\epsilon})$ bounds in a paper currently in preparation.

\section*{Acknowledgments}
S.S.~is grateful to Paul Jung for some useful conversations.
We are grateful to Nayantara Bhatnagar for discussing her result in \cite{BhatnagarPeled} with us, and especially for pointing out the 
importance of obtaining quantitative versions of a weak law.
We are grateful to Sumit Mukherjee for some useful discussions.
We are also appreciative of the 
helpful corrections and improvements of two anonymous referees.

\baselineskip=12pt
\bibliographystyle{plain}

\end{document}